\documentclass[11pt,reqno, twoside]{article}
\usepackage{mathrsfs}
\usepackage{amsmath,amssymb}
\usepackage{amsthm,bm}

\usepackage[colorlinks,linkcolor=red,anchorcolor=blue,citecolor=green]{hyperref}
\usepackage{graphics,graphicx}
\usepackage{cleveref}
\usepackage{color}
\usepackage{enumerate}
\usepackage[title]{appendix}
\usepackage{subfigure}
\usepackage{caption}
\usepackage{enumitem}
\usepackage{cite}
\setitemize[1]{itemsep=3pt,partopsep=0pt,parsep=\parskip,topsep=3pt}

\usepackage{todonotes}
\allowdisplaybreaks

\usepackage{tikz}
\usepackage{indentfirst}
\usepackage{titling}

\newtheorem{theorem}{Theorem}[section]

\newtheorem{lemma}[theorem]{Lemma}
\newtheorem{corollary}[theorem]{Corollary}
\newtheorem{proposition}[theorem]{Proposition}
\numberwithin{equation}{section}
\newtheorem{definition}[theorem]{Definition}
\theoremstyle{definition}
\newtheorem{remark}[theorem]{Remark}
\newtheorem{example}[theorem]{Example}
\newcommand{\runum}[1]{\romannumeral #1}

\def\R{{\mathbb{R}}}
\def\E{{\mathbb{E}}}
\def\P{{\mathbb{P}}}

\def\Z{{\mathbb{Z}}}
\def\N{{\mathbb{N}}}

\def\T{{\mathbb{T}}}
\def\d{{\mathrm{d}}}
\def\e{{\mathrm{e}}}
\def\X{{\mathcal{X}}}

\newcommand{\supp}{\operatorname{supp}}
\newcommand{\leb}{\operatorname{Leb}}

\title{Asymptotic stability and mean ergodicity of Feller processes on Polish spaces}
 \author{ \bf{\  Ziyu Liu$^a$,\   Jiehao Wan$^{b,*}$} \thanks{$^*$Corresponding author. Email:\; ziyu@ustb.edu.cn(Ziyu Liu), wanjiehao@stu.pku.edu.cn(Jiehao Wan).} \vspace{1mm}\\   
 \small \it  $^a$School of Mathematics and Physics, University of Science and Technology Beijing, 100083, Beijing, China\\
\small \it $^b$LMAM, School of Mathematical Science, Peking University, 100871, Beijing, China
 }
\date{\today}

\begin{document}
\maketitle
\begin{abstract}
This article establishes several necessary and sufficient criteria on asymptotic stability and mean ergodicity in various types of topologies for Feller processes taking values in Polish spaces. In particular, asymptotic stability and mean ergodicity in Wasserstein distance and weighted total variation distance are considered. The characterizations are formulated by using the notions of generalized eventual continuity properties and lower bound conditions, where the proofs invoke the coupling approach.
\vspace{5mm}


\vspace{5mm}
\noindent{\bf 2020 MR Subject Classification:} 60J25; 37A30
\end{abstract}

\tableofcontents

\section{Introduction}\label{Sec:Intro}
\renewcommand{\thethm}{\Alph{thm}}

Ergodic theory for stochastic dynamic systems is one of the fundamental topics in probability theory and plays a crucial role in understanding the long-term behavior of stochastic processes.
It has achieved fruitful theoretical results and found important applications in fields such as statistical mechanics, chemistry and biology. 
With the increasing expansion and deepening of its applications, for example, in the statistical theory of fluid mechanics (see \cite{HairerMattingly2006,BBPS2022a}), stochastic quantization (see \cite{RZZ2017,Tsatsoulis2018Weber,CCL2024})  and mixing properties for wave propagation in random media (see \cite{GMR2021,BFZ2023,LWX+2024}), it has also raised more theoretical issues that require further development in the ergodic theory of Markov processes.
Notions such as asymptotically strong Feller property, e-property, eventual continuity and approaches such as coupling, generalized coupling and functional inequality have emerged as the theory develops. On these related topics, readers may consult \cite{DaPratoZabczyk1996,MeynTweedie2009,Bakry2014,DMPS2018,Chen2005,Wang2005,GongWu2000,GongWu2006,LiuLiu2025} for more details. 

Among those topics, asymptotic stability and mean ergodicity, which refer to the transition probability kernels or their Ces\`aro averages converge toward the unique invariant probability measure under some given topology, stand as fundamental concepts for characterizing the asymptotic behavior of the process's distributions. Recently, in \cite{GLLL2024} Gong, Liu, Liu and the first author of this paper obtained equivalent characterizations of asymptotic stability and mean ergodicity in the sense of weak convergence for Markov-Feller semigroups in terms of (Ces\`aro) eventual continuity (EvC for brevity) and lower bound conditions (LBC for brevity), which are usually verifiable in specific models. In particular, eventual continuity, as a regularity of semigroups, is weaker than the well-known strong Feller property and e-property, thus easier to verify.

However, in the investigation of practical models and specific equations, apart from the ergodic behavior in the sense of weak convergence, researchers are also concerned with  the long-time convergence of processes under stronger topologies, which can provide richer information about the convergence of processes;  see e.g. \cite{Shao2015a,HairerMattingly2011b}. Therefore, a natural question arises: can we provide equivalent characterizations of asymptotic stability and mean ergodicity of Markov processes under stronger topologies? This will not only reveal deeper mathematical structures of the ergodicity of Markov processes but also provide more verifiable criteria for further applications in complex models.
This paper focuses on this question and derives equivalent characterizations of asymptotic stability and mean ergodicity in terms of generalized eventual continuity and lower bound conditions.

\subsection{Overview of main results}

Throughout this paper, we consider a Markov family $\left(\Omega, \mathcal{F}, \{\Phi_t\}_{t \ge 0}, \{\mathcal{F}_t\}_{t \ge 0}, \{\P^{x}\}_{x \in \X}\right)$, where $\Phi=\{\Phi_t\}_{t \ge 0}$ is a Feller process, right-continuous with left-hand limits, taking values in a Polish space $(\X, \rho)$. Denote by $\mathcal{P}(\X)$ the space of probability measures on $\X$, and let $\{P_t(\cdot, \cdot)\}_{t \ge 0}$ denote the transition kernels of $\Phi$. This paper investigates the \emph{asymptotic stability} and \emph{mean ergodicity} of $\Phi$ within a topological space $(\mathcal{P}(\X), \tau)$, as defined below.

\begin{definition}\label{def:AS_ME_tau}
Let $(\mathcal{P}(\X), \tau)$ be a topological space, where $\tau$ is a topology on $\mathcal{P}(\X)$.
    \begin{enumerate}[label=(\roman*)]
        \item $\Phi$ is called \textit{$\tau$-asymptotically stable} if there exists a unique invariant measure $\mu \in \mathcal{P}(\X)$ such that:
        \begin{equation*}
            \lim_{t \to \infty} P_t(x, \cdot) = \mu \quad \text{in } (\mathcal{P}(\X), \tau), \quad \forall\, x \in \X.
        \end{equation*}
        \item $\Phi$ is called \textit{$\tau$-mean ergodic} if there exists a unique invariant measure $\mu \in \mathcal{P}(\X)$ such that:
        \begin{equation*}
            \lim_{t \to \infty} \frac{1}{t} \int_0^t P_s(x, \cdot) \d s = \mu \quad \text{in } (\mathcal{P}(\X), \tau), \quad \forall\, x \in \X.
        \end{equation*}
    \end{enumerate}
\end{definition}

\vspace{2mm}

The main results of the present paper can be roughly described as follows. For a broad class of topologies $\tau$ on $\mathcal{P}(\X)$, there exists an associated family of test functions $\mathfrak{F}$ such that the following equivalent characterizations hold:
\begin{enumerate}[label=(\roman*)]
    \item 
    
    $\tau\text{-asymptotic stability}\iff \mathfrak{F}\text{-EvC + LBC \eqref{eq:LBC_P} }$;
    \item 
    $\tau\text{-mean ergodicity}\iff \mathfrak{F}\text{-Ces\`aro EvC + LBC \eqref{eq:LBC_Q}} $.
\end{enumerate} 
Here, LBC \eqref{eq:LBC_P} and LBC \eqref{eq:LBC_Q} refer to the lower bound conditions defined in Definition \ref{def:LBC}. $\mathfrak{F}$-(Ces\`aro) EvC refers to  generalized (Ces\`aro) eventually continuity  defined by test function $\mathfrak{F}$ (see Definition \ref{def:F_EvC}, \ref{def:F_Q_EvC}, \ref{def:F_uEvC} and \ref{def:F_Q_uEvC}). \vspace{2mm}

The precise statements are presented in Theorem~\ref{thm:f_mixing}, Theorem~\ref{thm:Q_f_mixing}, Theorem~\ref{thm:f_uniform_mixing}  and Theorem~\ref{thm:Q_f_uniform_mixing}. Important examples of topologies $\tau$ to which our main results apply include:
\begin{enumerate}[label=(\roman*)]
    \item the weak topology (e.g., \cite[Theorem 1]{GLLL2025}, \cite[Theorem 3.12]{GLLL2024});
    \item the topology induced by the Wasserstein distance (Theorem \ref{thm:W_mixing});
    \item the topology induced by the weighted total variation distance (Theorem \ref{thm:tv_mixing} and Theorem \ref{thm: V_mixing}).
\end{enumerate}

  Our strategy involves characterizing the topology $\tau$ through a family of test functions $\mathfrak{F}$ (see, for example, Lemma \ref{lem:equi_characterization_ofWasserstein_distance} and Kantorovich-Rubinstein duality principle in \cite[Theorem 5.1]{Villani2009}). A key difficulty arises when these test functions are unbounded, as the integrals may become ill-defined and the bounded convergence theorem is no longer applicable. To overcome this, we impose a uniform integrability condition to compensate for the lack of boundedness and integrability (see Hypotheses $\mathbf{(H_1)}$ and $\mathbf{(H_2)}$).
   
The proofs of our main results are based on a coupling approach. This method shares the same spirit in \cite{KulikScheutzow2015} and differs from the traditional proof originating from \cite{Lasota1977} and \cite{GLLL2025} based on contradiction. It is more applicable in the case of unbounded test functions since more probability tools and estimations can be used. The detailed arguments are developed in the subsequent sections.

    \subsection{Literature review }
    
    The existence and uniqueness of invariant measures and stability properties constitute basic problems in the ergodic theory of Feller processes.
    In general, the existence of invariant measures for Feller processes is typically established via Krylov-Bogoliubov theorem (see, e.g. \cite[Theorem 3.1.1]{DaPratoZabczyk1996}) and the key of the theorem is tightness of a family of the distributions evolving along with time. Concentration condition (see, e.g. \cite[Proposition 3.1]{LasotaSzarek2006}), built upon this theorem, aim to obtain existence of invariant measures by demonstrating that the mass of the process are concentrated over some compact set for all sufficiently large times. In locally compact space, this 
    can be easily established via Lyapunov functions. However, in non-locally compact spaces, identifying such a compact set becomes considerably more challenging.  A significant breakthrough in this direction emerged in 2006, when Lasota and Szarek developed lower bound techniques for Markov--Feller semigroups with the e-property to derive a sufficient condition for existence of invariant measures that is often more verifiable in infinite-dimensional settings (e.g., again through Lyapunov functions); see \cite{LasotaSzarek2006,Szarek2006}.

    Regarding uniqueness, in 1948, Doob developed the so called Doob's theorem to derive the uniqueness of invariant measure and asymptomatic stability in total variation distance. In 1960, Khas'minskii introduced the property of \emph{strong Feller}, using it to verify the regularity conditions required in Doob's theorem. Actually, strong Feller processes satisfy the ergodic measures are disjointly supported (EMDS for brevity) property. The EMDS property with suitable irreducibility conditions imply uniqueness of invariant measures. This strategy is now termed the Doob-Khas’minskii method. While the strong Feller property has successfully established unique ergodicity in many models, it often fails for stochastic differential equations with degenerate noise or when the state space of the process is infinite-dimensional; see e.g. the discussion in \cite[Example 3.15]{HairerMattingly2006}.
    In 2006, Hairer and Mattingly introduced a weaker regularity condition termed \emph{asymptotically strong Feller} in \cite{HairerMattingly2006}, successfully applying it to establish the uniqueness of invariant measure for two-dimensional stochastic Navier-Stokes equation on torus with highly degenerate additive Gaussian noise. A central component of their proof relies on the fact that asymptotically strong Feller processes possess the EMDS property. While the EMDS property is pivotal for proving the uniqueness of invariant measure, the asymptotically strong Feller condition is not necessary for deriving this property. For instance, processes with the e-property also exhibit the EMDS property; see, e.g., \cite[Theorem 1]{KSS2012} and \cite[Theorem 7.2.9]{Zaharopol2014}.

    Notably, the e-property lies at the intersection of the ingredients in proving both the existence and uniqueness of invariant measures, underscoring its significance for Markov processes. Extensive research has been conducted in this domain, including \cite{LasotaSzarek2006, Szarek2006,SSU2010,KPS2010,Worm2010,SzarekWorm2012}. This framework yields further extensions regarding asymptotic stability and mean ergodicity of Feller processes. In Worm's PhD thesis \cite{Worm2010}, it was established that for a Feller process $\Phi$ with the e-property, satisfying the lower bound condition \eqref{eq:LBC_Q} at some $z \in \X$ is equivalent to mean ergodicity under the weak topology. Subsequent developments by Szarek and Worm in \cite{SzarekWorm2012} demonstrated that for Feller process $\Phi$ with e-property, satisfying the condition \eqref{eq:LBC_P} at some $z \in \X$ characterizes asymptotic stability in weak topologies equivalently.
    
    However, the existence of Feller processes without the e-property motivates the study of such processes under weaker regularity conditions. Jaroszewska introduced asymptotic equicontinuity in \cite{Jaroszewska2013b} while Gong and Liu introduced eventual continuity in \cite{GongLiu2015}. These two concepts were proposed independently almost at the same time and are mathematically equivalent  which are strictly weaker than e-property (see \cite{LiuLiu2024} for an explicit counterexample). 
    Recent advances in this topic include works such as \cite{GLLL2025,GLLL2024,Liu2023,LiuLiu2024,LiLiu2025}. Readers may consult survey paper \cite{LiuLiu2025} for a quick overview.
    A recent development, as mentioned at the beginning of this paper is the equivalent characterization of asymptotic stability and mean ergodicity in weak topology in terms of (Ces\`aro) eventual continuity and the lower bound condition \eqref{eq:LBC_P} or \eqref{eq:LBC_Q}. 

    While progress has been made in understanding ergodicity under weak topologies with weak regularity conditions, the theory under stronger topologies remains less developed. The work presented in this paper represents a preliminary attempt to understand this question.

\vspace{3mm}

\subsection{Organization of the paper}
This paper is organized as follows. Section \ref{Sec: MainResult} presents the main results. Applications and examples are provided in Section \ref{Sec:3}, which include:
\begin{itemize}
    \item Equivalent characterizations of asymptotic stability in the Wasserstein distance and that for the stochastic Navier-Stokes equation with multiplicative noise on a bounded domain;
\item Equivalent characterizations of asymptotic stability and mean ergodicity in the (weighted) total variation distance and that for an explicit iterated function system lacking the e-property;
\item Comparison and extension of our main results to Feller processes with generalized e-property.
\end{itemize}

In Section \ref{Sec:Further discussions}, we discuss the necessity of the two conditions on the test function family $\mathfrak{F}$ in our main theorems, supported by two counterexamples.

Finally, Section \ref{Sec:4} details the proof strategies and technical proofs of the main theorems. Useful tools for the proofs, including properties of uniformly integrable random variables and a generalized Birkhoff's theorem, are provided in the Appendix.

\subsection{Notation}
Throughout this paper, we use $(\X,\rho)$ to denote the Polish space (i.e., a complete separable metric space endowed with metric $\rho$) and its Borel $\sigma$-algebra is denoted by $\mathcal{B}(\X)$.
 We define:\vspace{2mm}

$\begin{aligned}
	 \mathcal{P}(\X)&=\text{the family of probability measures on } \X,\\
    B(\X)&=\text{the family of Borel real-valued measurable functions on }\X,\\
    B_b(\X)&=\text{the family of bounded, Borel real-valued measurable functions on }\X,\\
    C(\X)&=\text{the family of continuous functions on }\X,\\
    C_c(\X)&=\text{the family of bounded continuous functions with compact support},\\
	L_b(\X)&=\text{the family of  bounded Lipschitz continuous functions, with the Lipschitz constant}\\
    &\quad\;\text{given by } \|f\|_{Lip}=\sup_{x,y\in \X,x\neq y}\,\frac{|f(x)-f(y)|}{\rho(x,y)},\\
    L^1(\mu)& =  \text{the family of integrable functions on } \X \text{ with respect to }\mu\in\mathcal{P}(\X),\\
    B(x,r)&=\{y\in \X:\rho(x,y)<r\} \text{ for } x\in \X \text{ and }r>0,\\
	\supp\mu&=\{x\in \X:\mu(B(x,\varepsilon))>0 \text{ for any } \varepsilon>0\}, \text{ for }\mu\in\mathcal{P}(\X), \\               &\quad\text{ i.e. the topological  support of the measure } \mu,\\
    \chi_{A}&= \text{indicator function of $A$},\\
        \N,\Z,\R,\R_+&= \text{natural numbers, integers, real numbers, nonnegative numbers, respectively}. 
	\end{aligned}$	
    
    \vspace{1mm}
    For brevity, we use the notation $\langle f,\mu\rangle=\int_{\X}f(x)\mu(\d x)$  for $f\in B(\X)$ and $\mu\in\mathcal{P}(\X)$ such that $\langle|f|,\mu\rangle<\infty$. For a random variable $\xi$, we denote its law by $\mathcal{D}(\xi)$.

\section{Main results}\label{Sec: MainResult}
 In this section, we present the main results of this paper. To begin with, let us introduce the following notions.  Recall that $\Phi$ denotes a Feller process on a Polish space $(\X,\rho)$. We denote $\{P_t\}_{t\geq 0}$ and $\{P_t^*\}_{t\geq 0}$ by its associated semigroups acting on $B_b(\X)$ and $\mathcal{P}(\X)$, respectively, given by
\begin{equation*}
     P_tf(x)=\int_{\X}f(y)P_t(x,\d y), \quad P_t^*\mu(A)=\int_\X P_t(x,A)\mu(\d x),
     \quad \forall\, f\in B_b(\X),\ \mu\in\mathcal{P}(\X).
 \end{equation*}

To accommodate a broader class of topologies as motivated in the introduction, we extend the domain of $\{P_t\}_{t\ge0}$ as follows. For each $t \ge 0$, define
\begin{equation*}
    \mathcal{A}_t=\left\{f\in B(\X):\int_\X|f(y)|P_t(x,\d y)<\infty,\ \forall\, x\in\X
\right\}, \quad \text{and set} \quad \mathcal{A}=\bigcap_{t\ge0}\mathcal{A}_t.
\end{equation*}
The semigroup $\{P_t\}_{t\ge0}$ can then be naturally extended to $\mathcal{A}$ by the same formula:
 \begin{equation*}
     P_tf(x)=\int_{\X}f(y)P_t(x,\d y), \quad \forall\, f\in\mathcal{A}.
 \end{equation*}
The Ces\`aro averages of the Markov semigroups and transition probabilities are denoted, respectively, by
\begin{equation*}
    Q_tf(x)=\frac{1}{t}\int_{0}^{t}P_sf(x)\d s, \quad Q_t(x,\cdot)=\frac{1}{t}\int_{0}^{t}P_s(x,\cdot)\d s, \quad \forall\, x\in\X,\ f\in\mathcal{A},\ t\geq 0.
\end{equation*}

As demonstrated by Lemma \ref{lem:equi_characterization_ofWasserstein_distance} and the Kantorovich--Rubinstein duality principle, convergence in certain topologies $\tau$ on $\mathcal{P}(\X)$ can be characterized by test functions equivalently. This paper focuses on such topologies. To be precise, we formally define the ergodic properties studied in this work as follows.

 \begin{definition}\label{def:mixing}
    Let $\mathfrak{F}\subset B(\X)$ be a family of measurable function.
    \begin{enumerate}[label=(\roman*)]
        \item $\Phi$ is \textit{asymptotically stable with respect to $\mathfrak{F}$} if it has a unique invariant measure $\mu\in\mathcal{P}(\X)$ such that $\mathfrak{F} \subset L^1(\mu)$ and for all $x\in\X$ and all $f\in\mathfrak{F}$,
        \begin{equation*}
        \lim\limits_{t\rightarrow\infty}\left|P_tf(x)-\langle f,\mu \rangle\right|=0.
        \end{equation*}
        
        \item $\Phi$ is \textit{mean ergodic with respect to $\mathfrak{F}$} if it has a unique invariant measure $\mu\in\mathcal{P}(\X)$ such that $\mathfrak{F} \subset L^1(\mu)$ and for all $x\in\X$ and all $f\in\mathfrak{F}$,
        \begin{equation*}
        \lim\limits_{t\rightarrow\infty}\left|Q_tf(x)-\langle f,\mu \rangle\right|=0.
        \end{equation*}
        
         \item $\Phi$ is \textit{uniformly asymptotically stable with respect to $\mathfrak{F}$} if it has a unique invariant measure $\mu\in\mathcal{P}(\X)$ such that $\mathfrak{F} \subset L^1(\mu)$ and for all $x\in\X$,
        \begin{equation*}
            \lim\limits_{t\rightarrow\infty}\sup_{f\in\mathfrak{F}}\left|P_tf(x)-\langle f,\mu \rangle\right|=0.
        \end{equation*}
         
         \item $\Phi$ is \textit{uniformly mean ergodic with respect to $\mathfrak{F}$} if it has a unique invariant measure $\mu\in\mathcal{P}(\X)$ such that $\mathfrak{F} \subset L^1(\mu)$ and for all $x\in\X$,
        \begin{equation*}
            \lim\limits_{t\rightarrow\infty}\sup_{f\in\mathfrak{F}}\left|Q_tf(x)-\langle f,\mu \rangle\right|=0.
        \end{equation*}
    \end{enumerate}
 \end{definition}
 
\begin{remark} \label{rem:example_for_F}
    Let us mention that the choice of $\mathfrak{F}$ can be quite general. We list here several specific cases for illustration.
    \begin{itemize}
        \item [(a)] $\mathfrak{F}=L_b(\X)$. In this case, the notion of eventual continuity with respect to $\mathfrak{F}$ coincides with the  eventual continuity defined in \cite[Definition 2.5]{GLLL2024}. Thus, Theorem \ref{thm:f_mixing} directly implies \cite[Theorem 3.16]{GLLL2024}.
        \item [(b)] $\mathfrak{F}=B_b(\X)$. In this case, the asymptotic stability with respect to $\mathfrak{F}$ is equivalent to that
        \begin{equation*}
        \lim\limits_{t\rightarrow\infty}P_t(x,A)=\mu(A)\quad\forall\, x\in\X,\;A\in\mathcal{B}(\X).
        \end{equation*}
        This conclusion aligns with Doob's theorem as stated in \cite[Theorem 4.2.1]{DaPratoZabczyk1996}. It is worth mentioning that \cite[Theorem 2.5]{LiLiu2025} proves a result that, in essence, corresponds to convergence in this topology. This result is also a consequence of Theorem~\ref{thm:f_mixing}.
        
        \item[(c)] $\mathfrak{F}=\left\{f\in C(\X):\sup_{x\in\X}\frac{|f(x)|}{1+\rho^p(x,x_0)}<\infty\right\}$ for some $x_0\in\X$ and $p\ge1$. Asymptotic stability with respect to this $\mathfrak{F}$ implies convergence in the Wasserstein distance (see Definition \ref{def:W_distance}).

        \item[(d)] $\mathfrak{F}=\left\{f\in B(\X):\sup_{x\in\X}\frac{|f(x)|}{1+V(x)}<\infty\right\}$ for some $V:\X\to\mathbb{R}_+$. Uniformly asymptotic stability with respect to this $\mathfrak{F}$ implies convergence in the $V$-weighted total variation distance (see Definition \ref{def:V_distance}).
       
    \end{itemize}               
    \end{remark}

The first component in the equivalent characterization of asymptotic stability and mean ergodicity with respect to $\mathfrak{F}$ is the lower bound condition. As introduced in Section~\ref{Sec:Intro}, such conditions play a central role in the ergodic theory of Markov processes (see, e.g., \cite{LasotaSzarek2006,Szarek2006}). More precisely, we introduce the following two types of lower bound conditions, which will be used throughout this paper.

\begin{definition}\label{def:LBC}
    
    \begin{itemize}[label=(\roman*)]
        \item[(1)] $\Phi$ is said to satisfy the lower bound condition \eqref{eq:LBC_P} at $z\in\X$ if 
        \begin{equation}
        \label{eq:LBC_P}
            \forall\, r>0,\quad\inf_{x\in \X}\liminf_{t\to \infty }P_t(x,B(z,r))>0.\tag{\(\mathcal{C}_1\)}
        \end{equation}
       
        \item[(2)] $\Phi$ is said to satisfy the lower bound condition \eqref{eq:LBC_Q} at $z\in\X$ if
        \begin{equation*}
        \label{eq:LBC_Q}
            \forall\, r>0,\quad\inf_{x\in \X}\limsup_{t\to \infty }Q_t(x,B(z,r))>0.
            \tag{\(\mathcal{C}_2\)}
        \end{equation*}
    \end{itemize}
\end{definition}

The second key concept is that of generalized (Ces\`aro) eventual continuity.

 \begin{definition}\label{def:F_EvC}$\Phi$ is said to be eventually continuous with respect to a family $\mathfrak{F}\subset\mathcal{A}$ at $x\in \X$ if for any $f\in\mathfrak{F}$,
		\begin{equation}\label{eq:f_EvC}
			\limsup\limits_{x'\rightarrow x}\limsup\limits_{t\rightarrow\infty}|P_tf(x')-P_tf(x)|=0.
		\end{equation} 
If \eqref{eq:f_EvC} holds for any $x\in\X$, then $\Phi$ is said to be eventually continuous with respect to $\mathfrak{F}$ on $\X$.
   
\end{definition}

In the definition above, the requirement $\mathfrak{F}\subset \mathcal{A}$ is necessary since it guarantees the well posedness of \eqref{eq:f_EvC}. To compensate the lack of boundedness of test function, a stronger assumption on test function $\mathfrak{F}$ is needed for deriving our main result:

\vspace{1mm}
\begin{itemize}
    \item[$\mathbf{(H_1)}$] \label{H_1}  For any $f\in\mathfrak{F}$ and $x\in\X$, $\{ f(\Phi_t)\}_{t\ge0}$ is uniformly integrable under $\P^x$.
\end{itemize}
\vspace{1mm}

 Our first characterization between asymptotic stability and semigroup regularities is formulated as follows.

\begin{theorem}\label{thm:f_mixing}
   Let $\mathfrak{F}\subset B(\X)$ be such that $L_b(\X)\subset\mathfrak{F}$ and satisfy Hypothesis $(\mathbf{H_1})$.
   Then the following statements are equivalent:
   \begin{itemize}
			\item[(\runum{1})] $\Phi$ is asymptotically stable with respect to $\mathfrak{F}$.
			\item[(\runum{2})] $\Phi$ is eventually continuous with respect to $\mathfrak{F}$ on $\X$, and there exists $z\in\X$ such that $\Phi$ satisfies the lower bound condition \eqref{eq:LBC_P} at $z$.
            \item[(\runum{3})] There exists $z\in\X$ such that $\Phi$ is eventually continuous with respect to $\mathfrak{F}$ at $z$ and satisfies the lower bound condition \eqref{eq:LBC_P} at $z$.
		\end{itemize}
\end{theorem}

Our next result concerns the asymptotic behavior of the Ces\`aro averages of the Feller process $\Phi$, which is another basic ergodic behavior of Feller processes (see \cite{Worm2010} and \cite{GLLL2024}).

\begin{definition}\label{def:F_Q_EvC}  $\Phi$ is said to be Ces\`aro eventually continuous with respect to a family $\mathfrak{F}\subset\mathcal{A}$ at $x\in \X$ if for any $f\in\mathfrak{F}$,
		\begin{equation*}
			\limsup\limits_{x'\rightarrow x}\limsup\limits_{t\rightarrow\infty}|Q_tf(x')-Q_tf(x)|=0.
		\end{equation*} 
\end{definition}

Building upon this concept, we establish the following result.
\begin{theorem}\label{thm:Q_f_mixing}
    Let $\mathfrak{F}\subset B(\X)$ be such that $L_b(\X)\subset\mathfrak{F}$ and satisfy Hypothesis $(\mathbf{H_1})$.
    Then the following statements are equivalent:
    \begin{itemize}
			\item[(\runum{1})] $\Phi$ is mean ergodic with respect to $\mathfrak{F}$.
			\item[(\runum{2})] $\Phi$ is Ces\`aro eventually continuous with respect to $\mathfrak{F}$ on $\X$, and there exists $z\in\X$ such that $\Phi$ satisfies the lower bound condition \eqref{eq:LBC_Q} at $z$.
            \item[(\runum{3})] There exists $z\in\X$ such that $\Phi$ is Ces\`aro eventually continuous with respect to $\mathfrak{F}$ at $z$ and satisfies the lower bound condition \eqref{eq:LBC_Q} at $z$.
		\end{itemize}
\end{theorem}

\begin{remark}
   
   Theorem~\ref{thm:Q_f_mixing} refines  \cite[Theorem 3.12]{GLLL2024}  by weakening the required semigroup regularity assumption to a localized version.  In practice, condition $(\runum{3})$ is more applicable, as it only requires verifying the Ces\`aro eventual continuity  at a single point $z$, rather than over the entire state space $\X$. 
\end{remark}

 The requirement ``$L_b(\X) \subset \mathfrak{F}$'' in Theorem \ref{thm:f_mixing} and Theorem \ref{thm:Q_f_mixing} can be replaced by the condition that ``$\{f \in L_b(\X) : \|f\|_\infty \le r\} \subset \mathfrak{F}$ for some $r > 0$''. This relaxation is justified by considering the scaled family $\bar{\mathfrak{F}} = \{\lambda f : \lambda \in \R,\ f \in \mathfrak{F}\}$ and noting that (Cesàro) eventual continuity, asymptotic stability, and mean ergodicity with respect to $\mathfrak{F}$ are equivalent to those with respect to $\tilde{\mathfrak{F}}$. When $\mathfrak{F} \subset B_b(\X)$, Hypothesis $\mathbf{(H_1)}$ is fulfilled naturally and we can omit it. 
 
 When $\mathfrak{F}$ contains unbounded functions, Hypothesis $\mathbf{(H_1)}$ is important. In particular, Section \ref{Subsec: H1_necessary} provide an explicit counterexample where violation of Hypothesis $\mathbf{(H_1)}$ induces failure of our main theorems. 
    On the other hand, this hypothesis is usually not difficult to verify and combining with Lemma~\ref{lem:boundedness_implies_UI}, a practical verification through Lyapunov function techniques is provided in Proposition~\ref{lem:S_Lypv}. This method can also be used to verify Hypothesis $\mathbf{(H_2)}$ below.

    \begin{proposition}\label{lem:S_Lypv}
    Let $\mathcal{L}$ be the generator of Feller process $\Phi$. If there exists a function $V:\X\to\R_+$ with  $\lim\limits_{\rho(x,x_0)\to\infty}V(x)=\infty$ for some $x_0\in\X$ and
    \begin{equation*}
        \mathcal{L}V(x)\le-\varphi(V(x))+C,
    \end{equation*}
    where $\varphi\in C^1(\R_+)$ is an  increasing nonnegative concave function, $C\ge0$. Then $$ \sup_{t\ge0}\E \varphi(V(\Phi_t))< \infty.$$ 
    \end{proposition}

The third main result of this paper establishes a connection between convergence in specific coupling distances on probability measures and the regularity properties of the associated semigroup. According to the Kantorovich--Rubinstein duality principle (see, e.g., \cite[Theorem 5.1]{Villani2009}), convergence in these distances is equivalent to uniform convergence over test functions. To address this, we introduce a uniform version of eventual continuity, derived by appropriately modifying Definition~\ref{def:F_EvC}.

\begin{definition}\label{def:F_uEvC}
    $\Phi$ is said to be \textit{uniformly eventually continuous with respect to a family $\mathfrak{F}\subset\mathcal{A}$ at $x\in\X$} if 
		\begin{equation*}
			\limsup\limits_{x'\rightarrow x}\limsup\limits_{t\rightarrow\infty}\sup_{f\in\mathfrak{F}}|P_tf(x')-P_tf(x)|=0.
		\end{equation*}
\end{definition}

It therefore follows that the notion of uniform eventual continuity is closely related to asymptotic stability with respect to coupling distances on probability measures. More specifically, we obtain the following result, which in turn implies asymptotic stability with respect to the ($V$-weighted) total variation distance; see Corollary~\ref{thm:tv_mixing} later.

\vspace{2mm}
We introduce the following hypothesis:
\begin{itemize}
    \item[$\mathbf{(H_2)}$] \label{H_2} There exists a function $V \in \mathcal{A}$ such that $|f(x)| \le V(x)$ for all $f\in\mathfrak{F}$ and $x\in\X$, and the family $\{ V(\Phi_t)\}_{t\ge0}$ is uniformly integrable under $\P^x$ for all $x\in\X$.
\end{itemize}

\begin{theorem}\label{thm:f_uniform_mixing}
    Let $\mathfrak{F}\subset B(\X)$ be such that $\{f\in L_b(\X):\|f\|_\infty\le 1\}\subset \mathfrak{F}$ and satisfy Hypothesis $\mathbf{(H_2)}$.
    Then the following statements are equivalent:
    \begin{itemize}
			\item[(\runum{1})] $\Phi$ is uniformly asymptotically stable with respect to $\mathfrak{F}$.
			\item[(\runum{2})] $\Phi$ is uniformly eventually continuous with respect to $\mathfrak{F}$ on $\X$, and there exists $z\in\X$ such that $\Phi$ satisfies the lower bound condition \eqref{eq:LBC_P} at $z$.
            \item[(\runum{3})] There exists $z\in\X$ such that $\Phi$ is uniformly eventually continuous with respect to $\mathfrak{F}$ at $z$ and satisfies the lower bound condition \eqref{eq:LBC_P} at $z$.
		\end{itemize}
\end{theorem}

Similarly, we can modify Definition \ref{def:F_Q_EvC} to a uniform version, analogous to the extension from Definition \ref{def:F_EvC} to Definition \ref{def:F_uEvC}, as follows.

\begin{definition}\label{def:F_Q_uEvC}
    $\Phi$ is said to be \textit{uniformly Ces\`aro eventually continuous with respect to a family $\mathfrak{F}\subset\mathcal{A}$ at $x\in\X$} if 
		\begin{equation*}
			\limsup\limits_{x'\rightarrow x}\limsup\limits_{t\rightarrow\infty}\sup_{f\in\mathfrak{F}}|Q_tf(x')-Q_tf(x)|=0.
		\end{equation*}
\end{definition}

Accordingly, we have the following theorem:

\begin{theorem}\label{thm:Q_f_uniform_mixing}
    Let $\mathfrak{F}\subset B(\X)$ be such that $\{f\in L_b(\X):\|f\|_\infty\le 1\}\subset \mathfrak{F}$ and satisfy Hypothesis $\mathbf{(H_2)}$.
    Then the following statements are equivalent:
    \begin{itemize}
			\item[(\runum{1})] $\Phi$ is uniformly mean ergodic with respect to $\mathfrak{F}$.
			\item[(\runum{2})] $\Phi$ is uniformly Ces\`aro eventually continuous with respect to $\mathfrak{F}$ on $\X$, and there exists $z\in\X$ such that $\Phi$ satisfies the lower bound condition \eqref{eq:LBC_Q} at $z$.
            \item[(\runum{3})] There exists $z\in\X$ such that $\Phi$ is uniformly Ces\`aro eventually continuous with respect to $\mathfrak{F}$ at $z$ and satisfies the lower bound condition \eqref{eq:LBC_P} at $z$.
		\end{itemize}
\end{theorem}

To conclude this section, we note that our results naturally extend to discrete-time Markov processes.

\section{Applications}\label{Sec:3}

\subsection{Asymptotic stability in $p$-Wasserstein distance}
 As a first application, we provide an equivalent characterization on asymptotic stability in the $p$-Wasserstein distance. 
 Recall that for $\mu,\nu\in\mathcal{P}(\X)$, a \emph{coupling} of $\mu,\nu$ is a probability measure $\pi$ on the product space $\X\times\X$ with marginals $\mu$ and $\nu$. We denote the set of all couplings between $\mu$ and $\nu$ by $\mathscr{C}{(\mu,\nu)}$. The $p$-Wasserstein distances is defined as follows.
\begin{definition}[{\cite[Definition 6.1]{Villani2009}}]\label{def:W_distance}
		Let $d$ be a semicontinuous metric on $(\X,\rho)$. For $p\in[1,\infty)$ \emph{$p$-Wasserstein distance for $d$} on $\mathcal{P}(\X)$ is
		\begin{equation*}
			W_{p,d}(\mu,\nu) := \inf_{\pi \in \mathscr{C}{(\mu,\nu)}}\left\{\int_{\X\times \X}d(x,y)^p\pi(\d x,\d y)\right\}^{1/p}.
		\end{equation*}
	\end{definition}
\noindent We note that when $d = \rho$, the distance $W_{p,\rho}$ reduces to the standard $p$-Wasserstein distance, denoted simply by $W_p$. Another classical example is the total variation distance, which will be discussed in the next subsection, obtained by setting $p = 1$ and $d(x, y) = \chi_{\{x \neq y\}}$. Convergence with respect to the $p$-Wasserstein distance can be characterized in terms of test functions, as detailed in the following lemma.

    \begin{lemma}[{\cite[Theorem 6.8]{Villani2009}}]
    \label{lem:equi_characterization_ofWasserstein_distance}
        Let $\{\mu_t\}_{t\ge0}\subset\mathcal{P}(\X)$ and $\mu\in \mathcal{P}(\X)$. Then the following statements are equivalent: 
    \begin{enumerate}[label=(\roman*)]
        \item$\lim\limits_{t\to\infty}W_{p}(\mu_t,\mu)=0$;
        \item For any  $\varphi\in C(\X)$ with $|\varphi(x)|\le 1+\rho(x_0,x)^p$, where $x_0\in\X$, one has
        \begin{equation*}
        \lim\limits_{t\to\infty}\langle \varphi,\mu_t\rangle =\langle \varphi,\mu\rangle.
    \end{equation*}
    \end{enumerate}
    \end{lemma}

   By applying Theorem \ref{thm:f_mixing} with the specific choice of test functions $\mathfrak{F}=\{f\in C(\X):|f(x)|\le (1+\rho(x_0,x))^p)\}$, we have:
    \begin{theorem}
		\label{thm:W_mixing}
		Let $\mathfrak{F} = \{ f \in C(\X) : |f(x)| \le 1 + \rho(x_0, x)^p \}$ with any given $x_0 \in \X$ be a family of test functions satisfying Hypothesis $\mathbf{(H_1)}$. Then the following statements are equivalent:
		\begin{itemize}
			\item [(\runum{1})] $\Phi$ admits a unique invariant measure $\mu\in\mathcal{P}(\X)$ 
            and for all $x\in\X$,
            \begin{equation*}
               \lim_{t\rightarrow\infty} W_p(P^*_t\delta_x,\mu)=0.
            \end{equation*}
			\item[(\runum{2})] $\Phi$ is eventually continuous with respect to $\mathfrak{F}$ on $\X$, and there exists $z\in\X$ such that $\Phi$ satisfies the lower bound condition \eqref{eq:LBC_P} at $z$.
            \item[(\runum{3})] There exists $z\in\X$ such that $\Phi$ is eventually continuous with respect to $\mathfrak{F}$ at $z$ and  satisfies the lower bound condition \eqref{eq:LBC_P} at $z$.
    \end{itemize}
\end{theorem}

Below we present an example involving the stochastic Navier-Stokes equation. As a direct application of Theorem \ref{thm:W_mixing}, we establish asymptotic stability in the $W_p$ distance for any $p \ge 1$. 
We believe the method employed in this example is applicable to many other models as well.
\begin{example}
Recall the  two-dimensional stochastic Navier-Stokes equation with multiplicative noise posed on a bounded domain \( D \subset \mathbb{R}^2 \) with a smooth boundary \( \partial D \):
\begin{equation*}\label{eq:SNS}\tag{SNS}
\begin{cases}
\d\mathbf{u} + \mathbf{u} \cdot \nabla \mathbf{u} \, \d t = (\nu \Delta \mathbf{u} - \nabla p) \d t + \sum_{k=1}^{m} \sigma_k(\mathbf{u}) \, \d W_t^k, \\
\mathbf{u}_0 = x, \quad \nabla \cdot \mathbf{u} = 0, \quad \mathbf{u}|_{\partial D} = 0,
\end{cases}
\end{equation*}
where \( \mathbf{u} = (u_1, u_2) \) is the unknown velocity field, \( p \) is the unknown pressure, \( m \in \mathbb{N} \), \( W = (W^1, \ldots, W^m) \) is a standard \( m \)-dimensional Brownian motion, \( \sigma_1, \ldots, \sigma_m : H \to H \) are measurable mappings, \( \nu > 0 \).

Consider \eqref{eq:SNS} on the phase space
\begin{equation*}
H := \{ \mathbf{u} \in L^2(D)^2 : \nabla \cdot \mathbf{u} = 0, \, \mathbf{u} \cdot \mathbf{n} = 0 \},
\end{equation*}
where \( \mathbf{n} \) is the outward normal to \( \partial D \). Denote \( P_L \) as the orthogonal projection of \( L^2(D)^2 \) onto \( H \). Define
\begin{equation*}
V := \{ \mathbf{u} \in H^1(D)^2 : \nabla \cdot \mathbf{u} = 0, \, \mathbf{u}|_{\partial D} = 0 \}.
\end{equation*}
The norms associated with \( H \) and \( V \) are denoted by \( |\cdot| \) and \( \|\cdot\| \), respectively. 

The Stokes operator is \( A\mathbf{u} = -P_L \Delta \mathbf{u} \) for \( \mathbf{u} \in V \cap H^2(D)^2 \). \( A \) is self-adjoint with compact inverse, so \( A \) has eigenvalues \( \lambda_k \sim k \) (diverging to infinity) with eigenvectors \( e_k \) forming a complete orthonormal basis for \( H \). Let \( P_N \) (projection onto \( H_N = \operatorname{span}\{e_k : k=1,\ldots,N\} \)) and \( Q_N \) (orthogonal complement).

We impose the following hypotheses:
\begin{itemize}
   \item[$\mathbf{(A_1)}$]  \( \sigma = (\sigma_1, \ldots, \sigma_m) \) is bounded and Lipschitz: \( \exists\, B_0, L > 0 \) with
  \begin{equation*}
  |\sigma(\mathbf{u})|^2 = \sum_{k=1}^m |\sigma_k(\mathbf{u})|^2 \leq B_0 \, (\forall\, \mathbf{u} \in H), \quad |\sigma(\mathbf{u}) - \sigma(\mathbf{v})|^2 \leq L |\mathbf{u} - \mathbf{v}|^2 \, (\forall\, \mathbf{u}, \mathbf{v} \in H).
  \end{equation*}
  \item[$\mathbf{(A_2)}$] \( \exists\, N \in \mathbb{N} \) with \( P_N H \subset \operatorname{Range}(\sigma_k(\mathbf{u})) \, (\forall\, \mathbf{u} \in H, \, k=1,\ldots,m) \). Pseudo-inverses \( \sigma_k^{-1} : P_N H \to H \) are uniformly bounded: \( \exists\, C_0 \) with \( |\sigma(\mathbf{u})^{-1}(P_N \mathbf{w})| \leq C_0 |P_N \mathbf{w}| \, (\forall\, \mathbf{u}, \mathbf{w} \in H) \).
   \item[$\mathbf{(A_3)}$]\( \lambda_N > \frac{L}{\nu} + \frac{C_0^2}{\nu^3} B_0 \).
\end{itemize}

Under these assumptions, \eqref{eq:SNS} has a unique strong solution (argued as in \cite{Odasso2008}) and it was shown in \cite{Odasso2008} that for any initial condition \( x \in H \) this equation has a unique strong solution, which in the case of ambiguity will be denoted later by \( \mathbf{u}^x \). To conclude, the equation \eqref{eq:SNS} defines a Feller process $\mathbf{u}$ on $H$. The main result of this example is stated as follows.
   \begin{proposition}
       Under the hypotheses $\mathbf{(A_1)}$-$\mathbf{(A_3)}$, the Feller process $\mathbf{u}$ admits unique invariant measure $\mu\in\mathcal{P}(\X)$. Moreover, for any $p\ge 1$ and $x\in H$,
            \begin{equation*}
               \lim_{t\rightarrow\infty} W_p(P_t^*\delta_x,\mu)=0.
            \end{equation*}
   \end{proposition}

\begin{proof}
The proof is divided into two parts, consisting of the verification of the lower bound condition and the eventual continuity.

\vspace{2mm}

\noindent\textbf{Lower bound condition.}
 We first establish a useful energy estimate, which serves as a Lyapunov functional for the system. Namely, we derive that there exists $\varepsilon>0$ such that for all $x\in H$
\begin{equation}\label{eq NS energy}
    \sup_{t\ge0}\E e^{\varepsilon\vert \mathbf{u}^x_t\vert^2}<\infty.
\end{equation}
 
 \vspace{1mm}
 
 To prove \eqref{eq NS energy}, using It\^{o}'s formula, one has
    \begin{align*}
        \d |\mathbf{u^x}|^2 =&-2\nu |\nabla\mathbf{u}^x|^2\d t +  2\sum_{k=1}^{m} \ \langle \mathbf{u}^x, \sigma_k(\mathbf{u}^x) \rangle\, \d W_t^k + \sum_{k=1}^{m} \ |\sigma_k(\mathbf{u}^x)|^2\d t\\
        \le&(-2\nu\lambda|\mathbf{u}^x|^2+B_0)\d t+2\sum_{k=1}^{m} \ \langle \mathbf{u}^x, \sigma_k(\mathbf{u}^x) \rangle\, \d W_t^k,
    \end{align*}
Furthermore, when $\varepsilon$ is sufficiently small, there exists $C_1,C_2>0$ such that 
    \begin{equation*}\label{eq:SNS_Ito}
        \begin{aligned}
            &{\d} e^{\varepsilon|\mathbf{u}^x|^2} \\
        ={} &\varepsilon e^{\varepsilon|\mathbf{u}^x|^2}
        \left(\d |\mathbf{u}^x|^2+ \frac{\varepsilon}{2}\d\langle |\mathbf{u}^x|^2 \rangle \right)\\
        ={}&\varepsilon e^{\varepsilon|\mathbf{u}^x|^2}\left[ -2\nu |\nabla\mathbf{u}^x|^2\d t +  2\sum_{k=1}^{m}  \langle \mathbf{u}^x, \sigma_k(\mathbf{u}^x) \rangle\, \d W_t^k + \sum_{k=1}^{m} |\sigma_k(\mathbf{u}^x)|^2\d t 
        + 2\varepsilon \sum_{k=1}^{m}  \left(\langle \mathbf{u}^x, \sigma_k(\mathbf{u}^x)\rangle\right)^2\d t \right]\\
        \le{}&\varepsilon \left[-2(\nu \lambda-\varepsilon B_0)|\mathbf{u}^x|^2 +B_0\right]e^{\varepsilon|\mathbf{u}^x|^2}\d t+\d M_t\\
        \le{}& C_1 -C_2 e^{\varepsilon|\mathbf{u}^x|^2}+ \d M_t,
        \end{aligned}
    \end{equation*}
    where $M_t=\sum_{k=1}^m\int_0^t\varepsilon e^{\varepsilon|\mathbf{u}^x(s)|^2}\langle\mathbf{u}^x(s),\sigma_k(\mathbf{u}^x(s))\rangle \d W_s^k$ is a local martingale.
    Therefore, there exists a sequence of nondeceasing random variables $\{T_n\}_{n=1}^\infty$ with $\lim\limits_{n\to\infty}T_n=\infty$ a.s. such that $M_{t\wedge T_n}$ is a martingale for all $n\in\N$. 
    Hence, proposition \ref{lem:S_Lypv}, this implies the desired inequality \eqref{eq NS energy}. 

Now given $p\geq 1$, let us consider the family 
\begin{equation*}
\mathfrak{F}=\{f\in C(H):|f(x)|\le 1+\|x\|^p\}.
\end{equation*} 
Note that for all $p\ge1$, there exists $C_p>0$ such that $\sup_{a\in\R}|a|^p/e^{a^2}\le C_p$. 
Combining Lemma \ref{lem:boundedness_implies_UI} with estimate \eqref{eq NS energy}, one sees that $\mathfrak{F}$ and the fact that satisfies Hypothesis $\mathbf{(H_1)}$.

Additionally, following the arguments as in \cite[Proposition 7.1]{Liu2023}, lower bound condition \eqref{eq:LBC_P} can be guaranteed by the energy estimate \eqref{eq NS energy} combined with a form of irreducibility. Noting that the associated irreducibility can be derived 
by similar arguments in \cite[Proposition 1.11]{DongPeng2024}, we therefore obtain the lower bound condition \eqref{eq:LBC_P}.

\vspace{2mm}
\noindent\textbf{Eventual continuity.} The second step is to derive the eventual continuity with respect to $\mathfrak{F}$. For \(f \in \mathfrak{F}\), denote
\begin{equation*}
    \underline{f}=f\chi_{\{|f|\le K\}}+K\chi_{\{f> K\}}-K\chi_{\{f<- K\}},\quad \overline{f} = f-\underline{f},
\end{equation*}
where $K\ge0$ will be determined later.
Hence,  $f =\underline{f}+\overline{f}$ and
    \begin{align*}
        &\lim_{y\to x}\limsup\limits_{t\to\infty}|P_t f(x) - P_t f(y)|\\
        \leq& \lim_{y\to x}\limsup\limits_{t\to\infty}|P_t \underline{f}(x) - P_t \underline{f}(y)|+\lim_{y\to x}\limsup\limits_{t\to\infty}|P_t \overline{f}(x)-P_t \overline{f}(y)|.
    \end{align*}
It is already proven that the system is eventually contiunous with respect to $L_b(H)$ in weak topology, see e.g. \cite[Proposition 7.6]{Liu2023}. Since $\underline{f}$ is a bounded Lipchitz function, we have
\begin{equation*}
    \lim_{y\to x}\limsup\limits_{t\to\infty}|P_t \underline{f}(x) - P_t \underline{f}(y)|=0
\end{equation*}
Let $V(x)=1+|x|^p$, then
\begin{equation*}
    | P_t \overline{f}(x)|\le \E\left[V(\mathbf{u}^x_t)\chi_{\{V(\mathbf{u}^x_t)>K\}}\right] \le  \E\left[V^2(\mathbf{u}^x_t)\right]^{1/2}\P(V(\mathbf{u}^x_t)\ge K)^{1/2}\le \frac{\sup_{t\ge0}\E\left[V^2(\Phi^x_t)\right]}{K}.
\end{equation*}

There exists $\tilde{C}>0$ independent of $x$ such that \begin{equation*}
\sup_{t\ge0}\E\left[V^2(\mathbf{u}^x_t)\right]\le1+2\sup_{t\ge0}E\left[|\mathbf{u}^x_t|^{p}\right]+\sup_{t\ge0}E\left[|\mathbf{u}^x_t|^{2p}\right]<\tilde{C}e^{\varepsilon|x|^2}.
\end{equation*}
Given $\varepsilon>0$, for $K\ge3\tilde{C}e^{\varepsilon|x|^2}/\varepsilon$
\begin{equation*}
     | P_t \overline{f}(x)|\le \varepsilon/3.
\end{equation*}
For this $\varepsilon$, there exists $r\ge0$ such that for any $y\in B(x,r)$, $\tilde{C}e^{\varepsilon|y|^2}\le 2\tilde{C}e^{\varepsilon|x|^2}$
\begin{equation*}
    | P_t \overline{f}(y)|\le \frac{\sup_{t\ge0}\E\left[V^2(\Phi^y_t)\right]}{K}\le\frac{2\tilde{C}e^{\varepsilon|x|^2}}{K}\le\frac{2\varepsilon}{3}.
\end{equation*}
To conclude, we get \begin{equation*}
    \lim_{y\to x}\limsup\limits_{t\to\infty}|P_t f(x) - P_t f(y)|\le \varepsilon.
\end{equation*}
Since $\varepsilon$ is arbitrary, 
\[
\limsup_{y\to x} \limsup_{t\to\infty} |P_t f(x) - P_t f(y)| = 0,
\]
proving eventual continuity w.r.t $\mathfrak{F}$ for $p\ge1$.

\end{proof}
\end{example}

\vspace{2mm}
\subsection{Asymptotic stability and mean ergodicity in (weighted) total variation distance}
At the beginning of this subsection, we provide an application of Theorem \ref{thm:f_uniform_mixing}, which further derives the equivalent characterization of asymptotical stability in total variation distance.\vspace{2mm}

To start, let $d$ be a bounded lower semi-continuous metric on $(\X,\rho)$. Define
    \begin{equation*}
        \|f\|_{d,Lip}:=\sup_{x\neq y}\frac{\left\vert f(x)-f(y)\right\vert}{d(x,y)}
    \end{equation*}
    Fix some $x_0\in\X$, define \begin{equation*}
        \mathfrak{F}_d:=\{f\in B(\X):f(x_0)=0,\|f\|_{d,Lip}\le1\}.
    \end{equation*} 
     Then for all $f\in\mathfrak{F}_d$, $|f(x)|=|f(x)-f(x_0)|\le d(x,x_0)$ and
        \begin{equation*}
            \|f\|_{\infty}\le \sup_{x\in\X}d(x,x_0)=M_0<\infty.
        \end{equation*}
    The last inequality is due to the boundedness of $d$. This implies that $\mathfrak{F}_d\subset\{f\in\mathfrak{F}:\|f\|_\infty\le M_0\}$ which ensures $\mathfrak{F}_d$ satisfies Hypothesis $\mathbf{(H_1)}$.
    By Kantorovich-Rubinstein duality principle (see e.g. \cite[Theorem 5.1]{Villani2009}), \begin{equation*}
        W_{d,1}(\mu,\nu)=\sup_{\|f\|_{d,Lip}\le1}|\langle f,\mu\rangle-\langle f,\nu\rangle|.
    \end{equation*}
    Since \begin{equation*}
        |\langle f+c,\mu\rangle-\langle f+c,\nu\rangle|=|\langle f,\mu\rangle-\langle f,\nu\rangle|,\quad \forall\, c\in\R,
    \end{equation*} we have
    \begin{equation}\label{eq:convergence_under_W_d1}
        W_{d,1}(\mu,\nu)=\sup_{f\in\mathfrak{F}_d}|\langle f,\mu\rangle-\langle f,\nu\rangle|.
    \end{equation}
    Hence uniformly asymptotic stability w.r.t $\mathfrak{F}_d$ is equivalent to asymptotic stability under $W_{d,1}$ distance. As a concrete application of Theorem~\ref{thm:f_uniform_mixing}, we have the following.
    
    \begin{theorem}
        Let $\mathfrak{F}_d$ satisfy $\{f\in L_b(\X):\|f\|_{\infty}\le r\}\subset\mathfrak{F}_d$ for some $r>0$. Then the following three statements are equivalent.
        \begin{itemize}
			\item [(\runum{1})] $\Phi$ admits a unique invariant measure $\mu\in\mathcal{P}(\X)$, and:
            \begin{equation}\label{eq:initial_distribution}
            \lim\limits_{t\to\infty}W_{d,1}(P_t^*v,\mu)=0,\quad \forall\nu\in \mathcal{P}(
            X
            ).
            \end{equation}
			\item[(\runum{2})] $\Phi$ is uniformly eventually continuous with respect to $\mathfrak{F}_d$ on $\X$, and there exists $z\in\X$ such that $\Phi$ satisfies the lower bound condition \eqref{eq:LBC_P} at $z$.
            \item[(\runum{3})] There exists $z\in\X$ such that $\Phi$ is uniformly eventually continuous with respect to $\mathfrak{F}_d$ at $z$ and satisfies the lower bound condition \eqref{eq:LBC_P} at $z$.
		\end{itemize}
    \end{theorem}
    \begin{proof}
        $\mathfrak{F}$ satisfies Hypothesis $(\mathbf{H_2})$ since for all $f\in\mathfrak{F}$, $\Vert f\Vert_\infty\le M_0$. 
        By Theorem~\ref{thm:f_uniform_mixing}, we have $(i)\implies(ii)$ and $(ii)\implies(iii)$. 
        
        When $(iii)$ holds true, according to Theorem~\ref{thm:f_uniform_mixing}, we have unique invariant measure $\mu\in\mathcal{P}(\X)$ such that for all $x\in\X$,
        \begin{align*}
            \lim_{t\ge0}\sup_{f\in\mathfrak{F}_d}| P_tf(x)-\langle f,\mu\rangle|=0.
        \end{align*}
        Note that \begin{align*}
            \limsup_{t\ge0}\sup_{f\in\mathfrak{F}_d}|\langle f,P_t^*f\rangle-\langle f,\mu\rangle|=&\limsup_{t\ge0}\sup_{f\in\mathfrak{F}_d}|\int_\X P_tf(x)-\langle f,\mu\rangle\nu(\d x)|\\
            \le &\limsup_{t\ge0}\int_\X\sup_{f\in\mathfrak{F}_d}| P_tf(x)-\langle f,\mu\rangle|\nu(\d x).
        \end{align*}
        Since
        \begin{equation*}
            \sup_{f\in\mathfrak{F}_d}| P_tf(x)-\langle f,\mu\rangle|\le2M_0,
        \end{equation*}
        we have 
        \begin{align*}
            \limsup_{t\ge0}\sup_{f\in\mathfrak{F}_d}|\langle f,P_t^*\nu\rangle-\langle f,\mu\rangle|            
            \le \int_\X\limsup_{t\ge0}\sup_{f\in\mathfrak{F}_d}| P_tf(x)-\langle f,\mu\rangle|\nu(\d x)=0.
        \end{align*}
        This combining \eqref{eq:convergence_under_W_d1} completes the proof.
    \end{proof}
    \begin{remark}
    In  \eqref{eq:initial_distribution}, the convergence is valid for any initial distribution. This statement relies on the boundedness of the test function family $\mathfrak{F}_d$. It is important to note that, in general, the convergence in part \runum{1} of our main results does not hold for arbitrary initial distributions. A counterexample validating this claim is presented in Section \ref{Subsec: H1_necessary}.
\end{remark}
   In particular, if we take $d(x,y)=\chi_{\{x=y\}}$, then $W_{d,1}$ distance is equivalent to total variation distance, which implies the following theorem.

    \begin{theorem}
		\label{thm:tv_mixing} 
		  Let $\mathfrak{F}=\{f\in B(\X):\Vert f\Vert_\infty\le1\}$. Then the following three statements are equivalent: 
		\begin{itemize}
			\item [(\runum{1})] There exists unique invariant probability measure $\mu$ for $\Phi$ such that $\forall\,\nu\in\mathcal{P}(\X)$, 
        \begin{equation*}
            \lim\limits_{t\to\infty}\|P_t^*\nu-\mu\|_{TV}=0.
        \end{equation*}
			\item[(\runum{2})] $\Phi$ is uniformly eventually continuous with respect to $\mathfrak{F}$ on $\X$, and there exists $z\in\X$ such that $\Phi$ satisfies the lower bound condition \eqref{eq:LBC_P} at $z$.
            \item[(\runum{3})] There exists $z\in\X$ such that $\Phi$ is uniformly eventually continuous with respect to $\mathfrak{F}$ at $z$ and satisfies the lower bound condition \eqref{eq:LBC_P} at $z$.
		\end{itemize}
	\end{theorem}

    Recall that Doob's theorem provides a specific sufficient condition for asymptotic stability in total variation distance. Actually, we can derive it from Theorem \ref{thm:f_uniform_mixing}:
    \begin{corollary}[{\cite[Theorem 1]{KulikScheutzow2015}}]
    \label{thm:Doob's_theorem}
        If stochastically continuous Markov process $\{\Phi_t\}_{t\ge0}$ is $t_0$-regular for some $t_0\ge0$ on Polish space $(\X,\rho)$ and has an invariant probability measure $\mu$, then $\mu$ is unique and all transition probabilities converge to $\mu$ in total variation distance.
    \end{corollary}
    \begin{proof}
        We verify the conditions  in $(iii)$ of Theorem \ref{thm:tv_mixing}. Due to the existence of invariant measure $\mu$, as in the proof of Theorem~\ref{thm:f_uniform_mixing} in Section~\ref{Sec:4}, lower bound condition \eqref{eq:LBC_P} hold at $z$ once $z\in\supp(\mu)$. 

        To verify the uniformly eventual continuity w.r.t $\mathfrak{F}$, we use coupling argument along the lines in \cite{KulikScheutzow2015}. \begin{equation*}
            \lim_{n\to\infty}\P(Z_n^1\neq Z_n^2)=0.
        \end{equation*}
        \begin{equation*}
            \sup_{f\in\mathfrak{F}}|P_nf(x)-P_nf(y)|
            \le\sup_{f\in\mathfrak{F}}\E|f(Z_n^1)-f(Z_n^2)|=\sup_{f\in\mathfrak{F}}\E|f(Z_n^1)-f(Z_n^2)\chi_{\{Z_n^1\neq Z_n^2\}}|\le2\P(Z_n^1\neq Z_n^2).
        \end{equation*}
    \end{proof}
    \vspace{3mm}

Weighted norm of probability is also frequently used in literature, especially for finite denominational systems \cite{MeynTweedie2009,HairerMattingly2011}. 

\begin{definition}\label{def:V_distance} Let  $V:\X\to\R_+$ be a continuous function. 
The \emph{$V$-weighted norm} of $\mu\in\mathcal{P}(\X)$ is defined by
    \begin{equation*}
        \Vert \mu\Vert_{V} := 1+\int_{\X}V(x)\mu(\d x)
    \end{equation*}
\end{definition}
Define the subspace of probability measure by
\begin{equation*}
    \mathcal{P}_V(\X):=\{\mu\in\mathcal{P}(\X):\Vert \mu\Vert_{V}<\infty\}
\end{equation*}
Note that we have the characterization 
\begin{equation*}
    d_V(\mu,\nu):=\Vert \mu-\nu\Vert_V= \inf_{\pi \in \mathscr{C}(\mu,\nu)}\int_{\{x\neq y\}}(2+V(x)+V(y))\pi(\d x,\d y),\quad \forall\,\mu,\nu\in\mathcal{P}_V(\X).
\end{equation*}

This implies that $d_V$ can be interpreted as a special case of $1$-Wasserstein distance defined in Definition \ref{def:W_distance} with $d(x,y) = \chi_{\{x\neq y\}}(2+V(x)+V(y))$.

With the preparation above, we provide an application of Theorem~\ref{thm:Q_f_uniform_mixing}, which derives an equivalent characterization of mean ergodicity under $d_V$.
    
    \begin{theorem}
        \label{thm: V_mixing} 
		  Let $\mathfrak{F}=\{f\in B(\X):|f(x)|\le1+V(x)\}$ satisfy Hypothesis $(\mathbf{H_2})$. Then the  following three statements are equivalent: 
		\begin{itemize}
			\item [(\runum{1})] $\Phi$ admits a unique invariant measure $\mu\in\mathcal{P}(\X)$ with $\mathfrak{F}\subset L^1(\mu) $ and for all $x\in\X$,
            \begin{equation*}
                \lim_{t\to\infty}d_V(Q_t^*\delta_x,\mu)=0.
            \end{equation*}
			\item[(\runum{2})] $\Phi$ is uniformly Ces\`aro eventually continuous with respect to $\mathfrak{F}$ on $\X$, and there exists $z\in\X$ such that $\Phi$ satisfies the lower bound condition \eqref{eq:LBC_Q} at $z$.
            \item[(\runum{3})] There exists $z\in\X$ such that $\Phi$ is uniformly Ces\`aro eventually continuous with respect to $\mathfrak{F}$ at $z$ and satisfies the lower bound condition \eqref{eq:LBC_Q} at $z$.
 		\end{itemize}
    \end{theorem}

As a concrete example, below we present an application on an iterated function system coupled with a periodic motion, originally constructed in \cite[Example 4.2]{GLLL2024}, which does not satisfy the Ces\`aro e-property. In particular, as a direct application of Theorem \ref{thm: V_mixing}, we establish the mean ergodicity in $V$-weighted total variation distance.

\begin{example}
    Let $\T=\R/2\pi$ be the $1$-D tours.
   The Feller process $\Xi=(\Phi,\Psi)$ on $\X=\R^+\times\T$ is given as follows. 
   For $\Phi$, let $\{\tau_n\}_{n\geq 1}$ be a sequence of random variables such that $\tau_0=0$ and $\triangle\tau_n=\tau_n-\tau_{n-1},n\geq 1$, are i.i.d. with density $\e^{-t}$. Define measurable mappings $w_i$ on $\R$ by
    \begin{equation*}
		w_1(x)=0,\quad w_2(x)=x,\quad w_3(x)=x^{-1}\chi_{\{x\neq 0\}},
    \end{equation*}
   and $p_i\in C(\X,[0,1])$ by
		\begin{equation*}
		(p_1(x),p_2(x),p_3(x))=\begin{cases}
		(\frac{x}{2},1-x,\frac{x}{2}),&0\leq x< \frac{2}{3},\\
		(\frac{1}{3},\frac{1}{3},\frac{1}{3}),&\frac{2}{3}\leq x\leq \frac{3}{2},\\
		(\frac{1}{2x},1-x^{-1},\frac{1}{2x}),&x>\frac{3}{2}.
		\end{cases}
		\end{equation*}
   With these settings, $\Phi=\{\Phi^x_t:x\in \X\}_{t\geq 0}$ is constructed in the following way. 

   \noindent\textbf{Step 1.} For any $x\in\X$, set $\hat{\Phi}_0^x:=x$.

   \noindent\textbf{Step 2.} Recursively, assume that $\hat{\Phi}_0^x,\dots,\hat{\Phi}_{n-1}^x,n\geq 1$ are given. We randomly choose $i_n\in \{1,2,3\}$ with probability $p_{i_n}(\hat{\Phi}_{n-1}^x),$ i.e,
	\begin{equation*}
	\P(i_n=k)=p_{k}(\hat{\Phi}_{n-1}^x)\quad\text{for } k\in \{1,2,3\}.
	\end{equation*}
	Then let $\hat{\Phi}_n^x=w_{i_n}(\hat{\Phi}_{n-1}^x).$

   \noindent\textbf{Step 3.} Now define 
	\begin{equation*}
	\Phi^x_t:=\hat{\Phi}_n^x\quad\text{for  }\tau_n\leq t<\tau_{n+1},\;n\geq 0.
	\end{equation*}
	  Next, we define $\Psi$ by
    $$
\Psi_t^y := y + t \pmod{2\pi}, \quad \text{for } y \in \mathbb{T}, \ t \geq 0.
$$
   \begin{proposition}
       Under the above settings, the Feller process $\Xi$ admits a unique invariant probability measure $\mu=\delta_0\times\leb(\T)$. Moreover, for any $V:\R_+\times\T\rightarrow[0,\infty)$ that is continuous,
            \begin{equation*}
               \lim_{t\rightarrow\infty} d_V(P_t^*\delta_x,\mu)=0, \quad \forall\, x\in\X.
            \end{equation*}
   \end{proposition}

   \begin{proof}
       From the discussion in \cite{GLLL2024}, the process $\Xi$ does not satisfy the Ces\`aro e-property at $0$,  $\Phi_t^x$ converges almost surely to $0$ as $t\rightarrow\infty$ for any $x\in\X$ and the lower bound condition \eqref{eq:LBC_Q} with $z=(0,0)$ holds true.

         It therefore suffices to verify Hypothesis $(\mathbf{H_2})$ and Ces\`aro eventual continuity at $(0,0)$ with respect to the family 
       \begin{equation*}
           \mathfrak{F}=\{f\text{ measurable}:f(x,y)\le1+V(x,y)\}.
       \end{equation*}

        Given $(x,y)\in\R_+\times\T$, since $
        \Phi_t\in\{0,x,x^{-1}\}$ and $V$ is continuous,
       \begin{equation*}  |f(\Phi_t,\Psi_t)|\le\sup_{y\in\T}V(0,y)+\sup_{y\in\T}V(x,y)+\sup_{y\in\T}V(x^{-1},y):=M_V<\infty. 
       \end{equation*}
       This implies that $\{|f(\Phi_t,\Psi_t)|\}_{t\ge0}$ is uniformly integrable.

        To get the Ces\`aro eventual continuity at $(0,0)$, let us fix $(x,y)\in\X$.
        \begin{align*}
        &\limsup_{t\rightarrow\infty}\sup_{f\in\mathfrak{F}}|Q_tf(x,y)-Q_tf(0,0)|\\
            \leq& \limsup_{t\rightarrow\infty}\sup_{f\in\mathfrak{F}}\E \left|\frac{1}{t}\int_0^tf(\Phi_s^x,\Psi_s^y)-f(\Phi_s^0,\Psi_s^0)\d s\right|\\
            \le&\limsup_{t\rightarrow\infty}\sup_{f\in\mathfrak{F}}\E \left|\frac{1}{t}\int_0^tf(\Phi_s^x,\Psi_s^y)-f(\Phi_s^0,\Psi_s^y)\d s\right|+\limsup_{t\rightarrow\infty}\sup_{f\in\mathfrak{F}}\E \left|\frac{1}{t}\int_0^tf(\Phi_s^0,\Psi_s^y)-f(\Phi_s^0,\Psi_s^0)\d s\right|\\
            =&I+I\!I.
        \end{align*}
        For the first term,
        \begin{align*}
            I\le&\limsup_{t\rightarrow\infty}\sup_{f\in\mathfrak{F}}\E \left|\frac{1}{t}\int_0^tf(\Phi_s^x,\Psi_s^y)-f(\Phi_s^0,\Psi_s^y)\chi_{\{\Phi_s\neq0\}}\d s\right|\\
            \le&\limsup_{t\rightarrow\infty}\E \frac{1}{t}\int_0^t\sup_{f\in\mathfrak{F}}\left|f(\Phi_s^x,\Psi_s^y)-f(\Phi_s^0,\Psi_s^y)\chi_{\{\Phi_s\neq0\}}\right|\d s.
        \end{align*}
        Note that\begin{equation*}
            \frac{1}{t}\int_0^t\sup_{f\in\mathfrak{F}}\left|f(\Phi_s^x,\Psi_s^y)-f(\Phi_s^0,\Psi_s^y)\chi_{\{\Phi_s\neq0\}}\right|\le2M_V<\infty.
        \end{equation*}
        By Fatou's lemma,
        \begin{equation*}
            I\le\E \limsup_{t\rightarrow\infty}\frac{1}{t}\int_0^t\sup_{f\in\mathfrak{F}}\left|f(\Phi_s^x,\Psi_s^y)-f(\Phi_s^0,\Psi_s^y)\chi_{\{\Phi_s\neq0\}}\right|\d s=0.
        \end{equation*}
        Since $\Phi_t^0\equiv0$, for the second term,
        \begin{align*}
            I\!I=&\limsup_{t\rightarrow\infty}\left|\frac{1}{t}\int_0^tf(0,\Psi_s^y)\d s-\frac{1}{t}\int_0^tf(0,\Psi_s^0)\d s\right|\\
            =&\limsup_{t\rightarrow\infty}\left|\frac{1}{t}\int_y^{t+y}f(0,\Psi_{s-y}^y)\d s-\frac{1}{t}\int_0^tf(0,\Psi_s^0)\d s\right|\\
            \le &\limsup_{t\rightarrow\infty}\frac{1}{t}\left\{\int_t^{t+y}V(0,\Psi_{s}^0)\d s+\int_0^yV(0,\Psi_s^0)\d s\right\}\\
            \le& \limsup_{t\rightarrow\infty}\frac{2y}{t}\cdot \sup_{y\in\T}V(0,y)=0.
        \end{align*}
        This completes the proof.

   \end{proof}
   \end{example}

\subsection{Asymptotic stability and mean ergodicity of equicontinuous processes}
     Equicontinuity, also known as the e-property, is an important regularity property for Markov process (see e.g. \cite{SzarekWorm2012}). Following Definition \ref{def:F_EvC}, \ref{def:F_Q_EvC}, \ref{def:F_uEvC} and \ref{def:F_Q_uEvC}, we can define (Ces\`aro) equicontinuity with respect to a collection of function $\mathfrak{F}$ in a similar way.
     
    \begin{definition}
         Let $\mathfrak{F}$ be a sub-family of $\mathcal{A}$.
    \begin{enumerate}[label=(\roman*)]
        \item $\Phi$ is said to be \emph{equicontinuous w.r.t $\mathfrak{F}$} at $x\in \X$ if 
		\begin{equation}
			\limsup\limits_{x'\rightarrow x}\sup\limits_{t\ge0}|P_tf(x')-P_tf(x)|=0,\quad \forall\,f\in \mathfrak{F}.
		\end{equation}
        \item $\Phi$ is said to be  \emph{uniformly equicontinuous w.r.t $\mathfrak{F}$} at $x\in \X$ if 
		\begin{equation*}
			\limsup\limits_{x'\rightarrow x}\sup\limits_{t\ge0}\sup_{f
            \in\mathfrak{F}}|P_tf(x')-P_tf(x)|=0.
		\end{equation*} 
    \end{enumerate}
    If $P_t$ is replaced by $Q_t$, then we refer to these properties as  \emph{Cesàro equicontinuous w.r.t.\ $\mathfrak{F}$} and \emph{uniformly Cesàro equicontinuous w.r.t.\ $\mathfrak{F}$} at $x \in \X$. If the property holds for all $x \in \X$, then we replace ``at $x \in \X$" by ``on $\X$" and sometimes omit this phrase for convenience.
    \end{definition}

    Since (Ces\`aro) equicontinuity is strictly stronger than (Ces\`aro) eventual continuity, we can apply Theorems~\ref{thm:f_mixing}, \ref{thm:f_uniform_mixing}, \ref{thm:Q_f_mixing} and \ref{thm:Q_f_uniform_mixing}
    
    to yield analogous results. we state the outcomes of Theorems~\ref{thm:f_mixing} and \ref{thm:Q_f_mixing} serve as examples; analogous extensions for the other theorems are omitted for brevity.
    \begin{corollary}\label{corollary:f_equi_mixing}
		Let $\mathfrak{F}\subset B(\X)$ be such that $L_b(\X)\subset \mathfrak{F}$ and satisfy Hypothesis $\mathbf{(H_1)}$. Suppose that $\Phi$ is equicontinuous w.r.t $\mathfrak{F}$ on $\X$, then the following statements are equivalent:
		\begin{itemize}
			\item [(\runum{1})] $\Phi$ is asymptotically stable w.r.t $\mathfrak{F}$ with unique invariant measure $\mu$.
			\item[(\runum{2})]  There exists $z\in\X$ such that $\Phi$ satisfies the lower bound condition \eqref{eq:LBC_P} at $z$.
		\end{itemize}
	\end{corollary}

        \begin{proof}
        Suppose (i) holds. Take $z\in \supp \mu$, then for all $x \in \X$, \begin{equation*}
        \lim\limits_{t\to\infty}P_t(x,B(z,r))\ge\mu(B(z,r))>0,
        \end{equation*}
        where the second inequality is justified by Portmanteau theorem.
        
        \noindent\textbf{(ii) $\Rightarrow$ (i):} Suppose (ii) holds true, we can apply Theorem \ref{thm:f_mixing} to get (i).
    \end{proof}

    \begin{corollary}
    		Let $\mathfrak{F}\subset B(\X)$ be such that $L_b(\X)\subset \mathfrak{F}$ and satisfy Hypothesis $\mathbf{(H_1)}$. Suppose that $\Phi$ is Ces\`{a}ro equicontinuous w.r.t $\mathfrak{F}$ on $\X$, then the following statements are equivalent: 
		\begin{itemize}
			\item [(\runum{1})] $\Phi$ is mean ergodic w.r.t $\mathfrak{F}$ with unique invariant measure $\mu$.
			\item[(\runum{2})]  There exists $z\in\X$ such that $\Phi$ satisfies the lower bound condition \eqref{eq:LBC_Q} at $z$.
		\end{itemize}
	\end{corollary}
    \begin{proof}
        Suppose (i) holds. As the proof in Corollary \ref{corollary:f_equi_mixing}, take $z\in \supp \mu$, then for all $x \in \X$, \begin{equation*}
        \lim\limits_{t\to\infty}Q_t(x,B(z,r))\ge\mu(B(z,r))>0.
        \end{equation*}
        
        \noindent\textbf{(ii) $\Rightarrow$ (i):} Suppose (ii) holds true, we can apply Theorem \ref{thm:Q_f_mixing} to get (i).
    \end{proof}
    \section{Further discussions}\label{Sec:Further discussions}
    In this section, we discuss the necessity of the conditions raised in the main result and discuss some hypotheses on the restriction.
    
    \subsection{One hypothesis on the restriction $L_b(\X)\subset\mathfrak{F} $}\label{Subsec:counterexample}

In Theorem~\ref{thm:f_mixing}, a key requirement on $\mathfrak{F}$ is that $L_b(\X) \subset \mathfrak{F}$, which ensures that the topology induced by the test function family $\mathfrak{F}$ is stronger than the weak topology. Indeed, once the inclusion $L_b(\X) \subset \mathfrak{F}$ is removed, it becomes difficult to guarantee the existence of an invariant measure (see Proposition~\ref{Prop:Exist*}).

One hypothesis on the restriction is that there exists one example, such that
\begin{equation*}
    \text{EvC w.r.t. } C_c(\X) + \text{LBC}\eqref{eq:LBC_P} \nRightarrow \text{ AS w.r.t. }C_c(\X).
\end{equation*}
Although we are not able to construct a precise counterexample that simultaneously satisfies the lower bound condition \eqref{eq:LBC_P} and the eventual continuity with respect to $C_c(\X)$ while failing to be asymptotically stable, there exists a related construction on this hypothesis. The following example, inspired by \cite{Szarek2006} with suitable modifications, demonstrates that if $\Phi$ is eventually continuous with respect to the family $\mathfrak{F} = C_c(\X)$, then a lower bound condition slightly weaker than \eqref{eq:LBC_P} is insufficient to guarantee the existence of an invariant measure. In other words, this indicates that the test function family $L_b(\X)$ is more suitable for studying asymptotic stability properties in Polish spaces.

\begin{example}
        
    Define a mapping $h: \mathbb{Z}_+ \times \mathbb{Z}_+ \times \bar{\mathbb{Z}}_+ \rightarrow l^\infty$ by
    \begin{equation*}
        h(i,j,k) \;=\; \big(i,\,\,\overbrace{0,\cdots,0}^{j\text{ terms}},\,\,2^{-k},\,\,0,\cdots\big),
    \end{equation*}
    where we adopt the convention that $2^{-\infty} = 0$. Let $\X$ denote the range of $h$, i.e., $\X
    =\{h(i,j,k):(i,j,k)\in\mathbb{Z}_+ \times \mathbb{Z}_+ \times \bar{\mathbb{Z}}_+\}$. Since $\mathbb{Z}_+ \times \mathbb{Z}_+ \times \bar{\mathbb{Z}}_+$ is countable,  $\X$ is countable as well.

    Consider a sequence $\{x_n=h(y_n,z_n,w_n)\}_{n \ge 1} \subset \X$. If $\lim\limits_{n\to\infty}x_n=x\in l^\infty$, then one of the following holds:
    \begin{itemize}
        \item[(1)] There exists $N \in \mathbb{N}$ such that $x_n = x$ for all $n \ge N$;
        \item[(2)] $\lim\limits_{n\to\infty}w_n=\infty$ and $x=(\xi,0,0,\cdots)=h(\xi,0,\infty)$ for some $\xi\in\Z$.
    \end{itemize}
    Consequently, $\X$ is a closed subspace of $l^\infty$ under the induced norm, and hence a Polish space.\vspace{2mm}

    We next define a Markov process $\Phi_n = h(\xi_n,\zeta_n,\eta_n)$ where $(\xi_n,\zeta_n,\eta_n)$ is a Markov process in $\mathbb{Z}_+ \times \mathbb{Z}_+ \times \bar{\mathbb{Z}}_+$. The transition probabilities of $(\xi_n,\zeta_n,\eta_n)$ is defined by:
    \begin{equation*}
        \begin{cases}
            P_{(i,j,k) \to (i,j+1,k+1)} \;=\; p_1(k), \\[6pt]
    
            P_{(i,j,k) \to (i+1,j+1,k)} \;=\;  p_2(i,k)  \\[6pt]
            P_{(i,j,k) \to (1,j+1,1)} \;=\; 1 - p_1(k)- p_2(i,k) .
        \end{cases}
    \end{equation*}

    \noindent\textbf{Step 1. $\Phi$ is a Feller process.}
    
    Let $\{x_n = m(y_n,z_n,w_n)\}_{n \ge 0} \subset \mathbb{X}$ be a sequence satisfying $\lim\limits_{n\to\infty}x_n=x$.
    Without loss of generality, assume $x = (\xi,0,\infty)$ for some $\xi\in\R$ and $y_n\equiv\xi$, $\lim\limits_{t\to\infty}w_n = \infty$. For the transition operator, we have:
    \begin{align*}
        P_1 f(x_n) &= p_1(w_n) \cdot f(m(\xi,z_n+1,w_n+1))+p_2(\xi,w_n) \cdot f(m(\xi+1,z_n+1,w_n))\\
        &\quad+ (1-p_1(w_n)-p_2(\xi,w_n)) \cdot f(m(1,z_n+1,1)) .
    \end{align*}
    By continuity of $f$, $\lim\limits_{n \to \infty} f(m(\xi_0,1,\eta_n+1)) = f(m(\xi_0,1,\infty))$. Thus:
    \[
        \lim\limits_{n \to \infty} P_t f(x_n) \;=\; f(m(\xi_0,1,\infty)) \;=\; P_t f(m(\xi_0,0,\infty)),
    \]
    confirming the Feller property.

    \vspace{1mm}
    
    \noindent\textbf{Step 2. Lower bounded condition.}
    
    From \cite[Section 4]{Szarek2006}, $\Phi$ satisfies the following lower bounded condition at $z=m(1,0,\infty)$: there exists $x\in\X$ for any $r\ge0$,
   \[
\limsup_{n\to\infty} \left( \frac{1}{n} \sum_{i=1}^n P^i(x, B(z,r)) \right) > 0 .
\]

    \noindent\textbf{Step 3. Eventual continuity at $z = m(1,0,\infty)$ with respect to $\mathfrak{F} = C_c(\X)$.}
    
    Compact subsets $K \subset \X$ take one of two forms:
    \begin{itemize}
        \item[(1)] $K$ is a finite set;
        \item[(2)] $K = \{x_n = m(\xi_n,\zeta_n,\eta_n)\}_{n \ge 0}$ with $x_0 = (\xi_0,0,\infty)$, and there exists $N \ge 1$ such that $\xi_n = \xi_0$ for all $n \ge N$ and $\lim\limits_{n \to \infty} \eta_n = \infty$.
    \end{itemize}
    
    For $f \in C_c(\X)$ with support of type (1): since $\zeta_n \to \infty$, we have $P_n f(x) \to 0$ for all $x \in \X$. Thus:
    \[
        \limsup_{n \to \infty} |P_n f(x) - P_n f(y)| \;\le\; \lim_{n \to \infty} |P_n f(x)| + \lim_{n \to \infty} |P_n f(y)| \;=\; 0.
    \]
    
    For $f \in C_c(\X)$ with support of type (2): there exists a sequence $\{y_m\}_{m \ge 1} \subset \X \setminus \supp(f)$ such that $\lim\limits_{m \to \infty} \|y_m - z\|_\infty = 0$. By continuity of $f$, this implies $f(z) = 0$. For any $\varepsilon > 0$, there exists $\delta > 0$ such that $|f(x)| < \varepsilon$ for all $x \in B(z,\delta)$. Note that $\supp(f) \setminus B(z,\delta)$ contains only finitely many elements. For sufficiently large $n$, $\P(\Phi_n \in \supp(f) \setminus B(z,\delta)) = 0$, so:
    \begin{equation*}
        \mathbb{E}f(\Phi_n) \;\le\; \varepsilon \cdot \P(\Phi_n \in B(z,\delta)) \;\le\; \varepsilon.
    \end{equation*}
    Since $\varepsilon$ is arbitrary, eventual continuity with respect to $C_c(\X)$ follows.

    \vspace{1mm}
    \noindent\textbf{Step 4. No invariant measure exists.}
    
    This follows from the fact that $\eta_n \to \infty$ almost surely, which would contradict Birkhoff's Ergodic Theorem if an invariant measure existed.
\end{example}

\subsection
{The necessity of the hypotheses on uniform integrability }\label{Subsec: H1_necessary}\label{Subsec:CE_U}

In this subsection, we present an example demonstrating that Hypothesis~$\mathbf{(H_1)}$ is, in some cases, necessary for $\mathfrak{F}$ in Theorem~\ref{thm:f_mixing}. Specifically, we construct a Feller process on a Polish space that satisfies the lower bound condition~\eqref{eq:LBC_P} and is eventually continuous with respect to a certain class of test functions $\mathfrak{F}\subset \mathcal{A}$, yet fails to satisfy Hypothesis~$\mathbf{(H_1)}$. We then verify that this process is not asymptotically stable with respect to $\mathfrak{F}$. The construction is described in detail below.

Consider a Markov process $\Phi$ on $\X=\{0\}\cup\{2^i:i\in\N\}$ with transition probabilities
\begin{equation*}
    p_{0,0}=1;\quad p_{2^i,0}=\frac{1}{2},\ p_{2^i,2^{i+1}}=\frac{1}{2},\quad\forall\, i\ge1.
\end{equation*}
Let $\rho$ be the metric on $\X$  induced from $\Z$, i.e. for any $x,y\in\X$, $\rho(x,y)=|x-y|$. Then $(\X, \rho)$ is a Polish space, and the Markov process $\Phi$ is Feller, since the space is discrete and has no accumulation points.

For all $n\ge1$, $p^{(n)}_{0,0}=1$ and when $i\ge1$,
\begin{align*}
&p^{(n)}_{2^i,2^{i+n}}=\P(\Phi_n=2^{i+n}|\Phi_0=2^i)=\frac{1}{2^n};\\
&p^{(n)}_{2^i,0}=\P(\Phi_n=0|\Phi_0=2^i)=1-\frac{1}{2^n}.
\end{align*}
Define $V(x) = x$ for all $x \in \X$ (i.e., the identity mapping on $\X$), then $P_nV$ is well defined for all $n\in\N$ and
\begin{equation}
\label{eq:ex2_boundedness}
    P_nV(2^i)=\int_\X V(y)p_n(2^i,\d y) =\E(\Phi_n|\Phi_0=2^i)=2^i<\infty.
\end{equation}
 Consider the test function family
\begin{equation*}
    \mathfrak{F}_\alpha=\{f\in C(\X):f(x)\le V^\alpha(x),\ \forall\, x\in\X\}, \quad  \alpha\in(0,1]. 
\end{equation*}
 
 The eventual continuity w.r.t $\mathfrak{F}_\alpha$ at $0$ is satisfied since $\X$ is discrete. $\Phi$ satisfies the lower bound condition \eqref{eq:LBC_P} at $z=0$ since \begin{equation*}
     \lim\limits_{n\to\infty}\P(\Phi_n=0|\Phi_0=x)=1,\quad\forall\, x\in\X.
 \end{equation*}
  By \eqref{eq:ex2_boundedness}, $\mathfrak{F}_{\alpha}$ satisfies Hypothesis $\mathbf{(H_1)}$ for $0<\alpha<1$. By Theorem \ref{thm:f_mixing}, we can conclude that $\Phi$ is asymptotically stable w.r.t. $\mathfrak{F}_\alpha$ with $0<\alpha<1$.
  In particular, the unique invariant measure for $\Phi$ is $\delta_0$.
  
Now consider $\alpha=1$. Given any $K\ge0$, for $n\ge \log_2K$, \begin{equation*}
    \chi_{\{|\Phi_n|\ge K\}}=\chi_{\{\Phi_n\neq 0\}}. 
\end{equation*}
Thus, when $x\neq0$ \begin{equation*}
    \limsup\limits_{n\to\infty}\E(|\Phi_n|\chi_{\{|\Phi_n|\ge K\}}|\Phi_0=x)=\limsup\limits_{n\to\infty}\E(\Phi_n|\Phi_0=x)=x>0.
\end{equation*}
This implies that $\{V(\Phi_n)\}_{n \ge 0}$ is not uniformly integrable; that is, Hypothesis $\mathbf{(H_1)}$ is not satisfied for $\mathfrak{F}_1$.

 Meanwhile, for the initial distribution $\nu = \delta_1 \neq \delta_0$,
\begin{equation*}
    \langle V, P^*_t \delta_1 \rangle = \E[V(\Phi_n) \mid \Phi_0 = 1) = 1 \neq \langle V, \delta_0 \rangle.
\end{equation*}
This shows that asymptotic stability with respect to $\mathfrak{F}_1$ fails.

Finally, we emphasize that asymptotic stability can not holds for initial distribution $\nu$ with $\mathfrak{F}\not\subset L^1(\nu)$. As an example, take $\alpha=1/2$ and $\nu =(p_k)_{k=1}^\infty$, where \begin{equation}
    p_k=\nu(2^k)=\left\{\begin{aligned}
        &0,   & \text{if }k\text{ is odd or }k=0;\\
        &c\cdot \frac{1}{m^2}, &\text{if }k=2m \text{ with }m\in\N_+;
    \end{aligned}\right.
\end{equation}
with $c=\left(\sum_{m=1}^\infty\frac{1}{m^2}\right)^{-1}$.
Then $\langle V^{1/2},\nu\rangle=\sum_{m=1}^\infty c\cdot\frac{1}{m^2}\cdot2^m=\infty.$

Define $P_n^*\nu=(p_k^{(n)})_{k=1}^\infty$ where $p_k^{(n)}=P_n^*\nu(2^k)=\P(\Phi_n=2^k|\mathcal{D}(\Phi_0)=\nu)$. Then \begin{equation}
    p_k^{(n)}=\left\{\begin{aligned}
        &0,   & \text{ if $k$ is odd, or if $k$ is even and } k \le n;\\
        &c\cdot \frac{1}{2^n}\cdot\frac{1}{m^2}, &\text{ if }k=2m+n \text{ with }m\in\N_+;
    \end{aligned}\right.
\end{equation}
Then\begin{equation*}
     \langle V^{1/2},P_n^*\nu\rangle=\sum_{m=1}^\infty c\cdot\frac{1}{2^n}\cdot\frac{1}{m^2}\cdot\sqrt{2^{2m+n}}=\frac{c}{2^n}\sum_{m=1}^\infty \frac{2^m}{m^2}=\infty.
\end{equation*}
Hence, $\langle V^{1/2},P_n^*\nu\rangle$ can not have finite limits.

\section{Proofs}\label{Sec:4}
\subsection{Sketch of the proofs}

This section is mainly devoted to the proofs of Theorems~\ref{thm:f_mixing}, \ref{thm:Q_f_mixing}, \ref{thm:f_uniform_mixing}, and \ref{thm:Q_f_uniform_mixing}. Before presenting the detailed arguments in the following subsections, we briefly outline the main ideas of the proofs. Among these results, the most technical part lies in establishing the implications $(\runum{3}) \Rightarrow (\runum{1})$ in each part. Since the proofs of these implications follow essentially the same strategy, we illustrate the core ideas by sketching the proof of this step in Theorem~\ref{thm:f_mixing}.

\vspace{2mm}

To begin, we recall a useful result originating from \cite{LasotaSzarek2006,Szarek2006,GongLiu2015,GLLL2024}, which guarantees the existence of an invariant measure under weak regularity properties together with the lower bound condition.

\begin{proposition}\label{Prop:Exist*}
		Suppose that there exists some $z\in \X$ such that $\Phi$ is Ces\`aro eventually continuous at $z$ and for any $\varepsilon>0$, 
		\begin{equation}\label{eq:LBC_0}\tag{$\mathcal{C}_3$}
			\limsup\limits_{t\rightarrow\infty}Q_t(z,B(z,\varepsilon))>0.
		\end{equation}	
		Then $\{Q_t(z,\cdot)\}_{t\ge0}$ is tight, which implies that there exists an invariant measure for  $\Phi$.
	\end{proposition}
	\begin{remark}
		In \cite[Theorem 3.2]{GLLL2024}, it requires  $\Phi$ to be Ces\`aro eventually continuous on $\X$, which is stronger than being Ces\`aro eventually continuous at $z$. However, a closer examination of the proof reveals that the core arguments remain valid and the requirement can be weaken in the form stated in this proposition without any modification.
	\end{remark}
    
Since $L_b(\X)\subset\mathfrak{F}$, combining Proposition \ref{Prop:Exist*}, Gong, Liu, Liu and Liu
has shown the existence and uniqueness of invariant measure $\mu$ in \cite[Theorem 3.12]{GLLL2024}. Therefore, the key is to show:
\begin{equation}\label{eq:goal}
    \lim_{t\to\infty}|P_tf(x)-\langle f,\mu\rangle|=0, \quad \forall\, f\in\mathfrak{F}.
\end{equation}
The first issue to address is that $\langle f,\mu\rangle$ may be ill-defined when $f$ is unbounded.
 However, invoking the generalized Birkhoff's ergodic theorem (Lemma \ref{lem:generalized_Birk}) and Hypothesis $\mathbf{(H_1)}$, the integrability of $f$ can be assured (see Lemma \ref{lem:integrability_by_generalised_Bir} for more details).

Next, to establish \eqref{eq:goal}, it suffices to  construct a  coupling $(\Phi^{(1)},\Phi^{(2)})$ process such that 
\begin{equation*}
     \lim_{t\to\infty}|\E^xf(\Phi_t)-\E^\mu f(\Phi_t)|=0.
\end{equation*}
The key step then turns to the construction an independent coupling $(\Phi^{(1)},\Phi^{(2)})$ on a same probability space, such that the law of $\Phi^{(1)}$ and $\Phi^{(2)}$ coincides with $\Phi$ under $\P^x$ and $\P^\mu$, respectively. Let us rewrite \eqref{eq:goal} as:
\begin{equation}\label{eq:goal2}
\lim_{t\to\infty}\left|\E^{\delta_x\times \mu}f(\Phi_t^{(1)})- f(\Phi_t^{(2)})\right|=0
\end{equation}

In view of the eventual continuity of $\Phi$, this implies that for any $\varepsilon>0$, there exists $r>0$ such that if $x\in B(z,r)$, 
\begin{equation*}
    \limsup_{t\to\infty}\left|\E^{(x,z)}f(\Phi_t^{(1)})- f(\Phi_t^{(2)})\right|\le\varepsilon.
\end{equation*}
It therefore ensures that for any $\Theta\in\mathcal{P}(\X\times \X)$ on $\X\times\X$ satisfying $\supp\Theta\subset B(z,r)\times B(z,r)$, one has
\begin{equation*}
    \limsup_{t\to\infty}\left|\E^{\Theta}f(\Phi_t^{(1)})- f(\Phi_t^{(2)})\right|\le\varepsilon.
\end{equation*}
This motivates us to use such construction to verify \eqref{eq:goal2}. However, such relation holds only when the initial distribution is supported on $ B(z,r)\times B(z,r)$.

To overcome this difficulty, we employ the lower bound condition. More precisely, this condition ensures that, almost surely, the two components of the coupling process, $\Phi^{(1)}$ and $\Phi^{(2)}$, visit the ball $B(z,r)$ simultaneously for any initial distribution $\Theta_0$, in particular for $\Theta_0 = \delta_x \times \mu$. Roughly speaking, this fact, combined with the strong Markov property, leads to our main results.

The proof, however, involves certain technical details concerning the interchange of the limit and integral operations. When the integrand is bounded, this step follows easily from the bounded convergence theorem. However, when $f$ is unbounded, additional care is required. To address this, we employ the notion of uniform integrability to compensate for the lack of boundedness (see Lemma~\ref{lem:UI_exchange_limit_and_integral}), thereby resolving the issue.

The complete proofs are presented in the following subsections. For the reader’s convenience, we also recall in Appendix~\ref{appendix:UI} some basic properties of uniformly integrable random variables, which play a key role in the proof of the main theorems.

\subsection{Proof of Theorem \ref{thm:f_mixing}}\label{Subsection:proof_of_f_mixing}
\begin{proof}
	\noindent\emph{(\runum{1}) $\implies$ (\runum{2}):}
	By asymptotic stability w.r.t $\mathfrak{F}$, there exists a unique invariant measure $\mu\in\mathcal{P(\X)}$ such that $\mathfrak{F}\subset L^1(\mu)$ and for any $x \in \X$, $f\in\mathfrak{F}$, 
    \begin{equation*}
        \lim\limits_{t\to\infty}\vert P_tf(x) - \langle f, \mu \rangle \vert=0.
    \end{equation*}
    
    Given $x,y\in\X$, $f\in\mathfrak{F}$,
\begin{equation*}
		\lim\limits_{t\to\infty}\vert P_tf(x) - P_tf(y)\vert 
        \le \lim\limits_{t\to\infty}\vert P_tf(x) - \langle f, \mu \rangle \vert  + \lim\limits_{t\to\infty}\vert P_tf(y) - \langle f, \mu \rangle \vert =0 ,
	\end{equation*}
	which implies eventual continuity w.r.t $\mathfrak{F}$ on $\X$.
		
	For $z \in \supp \mu$, $r>0$, since $\mathfrak{F}\supset L_b(\X)$, there exists $f \in \mathfrak{F}$ such that $\chi_{B(z,r/2)} \le f \le \chi_{B(z,r)} $. Hence for all $x \in \X$, 
	\begin{equation*}
		\lim\limits_{t\to \infty} P_t (x, B(z,r)) \ge \lim\limits_{t\to \infty} P_tf(x) = \langle f, \mu \rangle \ge \mu(B(z,r/2)) > 0.
	\end{equation*}

	\noindent\emph{(\runum{2}) $\implies$ (\runum{3}):} This implication is straightforward.\vspace{2mm}	
			
	\noindent\emph{(\runum{3}) $\implies$ (\runum{1}):} Since \eqref{eq:LBC_P} implies \eqref{eq:LBC_0} and $L_b(\X)\subset\mathfrak{F}$, using Proposition \ref{Prop:Exist*}, there exists an ergodic invariant  measure  $\mu\in \mathcal{P}(\X)$.

    Firstly,  we show $\mathfrak{F}\subset L_1(\X,\mu)$, which is directly deduced from the following lemma. 
     \begin{lemma}\label{lem:integrability_by_generalised_Bir}
         Let $\mu\in \mathcal{P}(\X)$ be an ergodic invariant measure for $\Phi$ and $f\in B(\X)$. If there exists $\bar{\X}\in\mathcal{B}(\X)$ with $\mu(\bar{\X})>0$ and for $x\in \bar{\X}$
         \begin{equation*}
             \sup_{t\ge0}\E^x|f(\Phi_t)|<\infty,
         \end{equation*}then $f\in L^1(\mu)$.
     \end{lemma}
    \begin{proof}
        We argue by contradiction. Otherwise, assume that $\langle |f|,\mu\rangle=\infty$. By Lemma~\ref{lem:generalized_Birk}, setting \begin{equation*}
            A=\left\{\omega:\lim\limits_{t\to\infty}\frac{1}{t}\int_0^t|f(\Phi_s(\omega))|\d s=\infty\right\},
        \end{equation*} then
        \begin{equation*}
     \P^\mu(A)=\int_\X\P^x(A)\mu(\d x)=1,
        \end{equation*}
 which implies that
 \begin{equation*}
     \P^x(A)=1\quad\mu\text{-a.s. }x\in\X. 
 \end{equation*} 
        To conclude, there exists $\X'\subset\X$ such that for all $x\in\X'$, $\P^x$-a.s. $\omega$,
    \begin{equation*}
    \lim\limits_{t\to\infty}\frac{1}{t}\int_0^t|f(\Phi_s(\omega))|\d s=\infty.
    \end{equation*}
    By Fatou's lemma,
    \begin{equation*}
        \E^x \lim\limits_{t\to\infty}\frac{1}{t}\int_0^t|f(\Phi_s(\omega))|\d s \le \liminf\limits_{t\to\infty} \E^x\frac{1}{t}\int_0^t|f(\Phi_s(\omega))|\d s=\liminf\limits_{t\to\infty} Q_t|f|(x),
    \end{equation*}
    which contradicts to the fact that $\sup_{t\ge0}Q_t|f|(x) \le \sup_{t\ge0}P_t|f|(x)<\infty$ for all $x\in\X$.
    \end{proof}

    To complete the proof of Theorem \ref{thm:f_mixing}, we shall apply Lemma \ref{lemma:uniformly_loss_of_memory} below. To this end, recall that by our assumption, Hypothesis $(\mathbf{H_1})$ holds for a family $\mathfrak{F}\subset B(\X)$. Then for any $f\in\mathfrak{F}$, let us consider the subfamily $\tilde{\mathfrak{F}}:=\{f\}\subset\mathfrak{F}$. Clearly, under such definition, Hypothesis $(\mathbf{H_2})$ is satisfied for $\tilde{\mathfrak{F}}$. Thus, using  Lemma \ref{lemma:uniformly_loss_of_memory} we conclude that,
        \begin{equation*}
			\limsup\limits_{t \to \infty}\sup_{f\in\tilde{\mathfrak{F}}} |P_tf(x) - \langle f,\mu\rangle | = \limsup\limits_{t \to \infty} |P_tf(x) - \langle f,\mu\rangle | = 0,\quad \forall\, x \in \X,
 	\end{equation*}
    which completes the proof.
    \end{proof}

\subsection{Proof of Theorem \ref{thm:Q_f_mixing}}\label{Subsection:proof_of_Q_f_mixing}
\begin{proof}
    \noindent\emph{(\runum{1})$\implies$(\runum{2}):}
	By mean ergodicity w.r.t $\mathfrak{F}$, there exists a unique invariant measure $\mu\in\mathcal{P(\X)}$ such that $\mathfrak{F}\subset L^1(\mu)$ and for any $x \in \X$, $f\in\mathfrak{F}$, 
    \begin{equation*}
        \lim\limits_{t\to\infty}\vert Q_tf(x) - \langle f, \mu \rangle \vert=0.
    \end{equation*}
    Given $x,y\in\X$,
	\begin{equation*}
		\lim\limits_{t\to\infty}\vert Q_tf(x) - Q_tf(y)\vert 
        \le \lim\limits_{t\to\infty}\vert Q_tf(x) - \langle f, \mu \rangle \vert  + \lim\limits_{t\to\infty}\vert Q_tf(y) - \langle f, \mu \rangle \vert =0 ,
	\end{equation*}
	which implies Ces\`aro eventually continuity w.r.t $\mathfrak{F}$ on $\X$.\vspace{1mm}
		
	Take $z \in \supp \mu$. Given any $r>0$, since $L_b(\X)\subset\mathfrak{F}$, there exists $f \in \mathfrak{F}$ such that $\chi_{B(z,r/2)} \le f \le \chi_{B(z,r)} $. Hence for all $x \in \X$, 
	\begin{equation*}
		\lim\limits_{t\to \infty} Q_t (x, B(z,r))  \ge \lim\limits_{t\to \infty} Q_tf(x) = \langle f, \mu \rangle \ge \mu(B(z,r/2)) > 0.
	\end{equation*} 
    This verifies the lower bound condition \eqref{eq:LBC_Q} at $z$.
    \vspace{1mm}
    
	\noindent\emph{(\runum{2})$\implies$(\runum{3}):} This implication is straightforward.\vspace{2mm}
    
	\noindent\emph{(\runum{3})$\implies$(\runum{1}):} The proof is divided into two steps.
    
    \vspace{2mm}
    \noindent 	$\mathbf{Step\;1}.$ We first show there exists $z\in\X$ such that for any $\varepsilon>0$,
            \begin{equation}\label{eq:LBC_3}
                \inf_{x\in\X}\liminf\limits_{t\rightarrow\infty}Q_t(x,B(z,\varepsilon))>0,\tag{$\mathcal{C}_3$}
            \end{equation}
    This step is nearly the same of that in \cite{GLLL2024} but only requiring $\{f\in L_b(\X):\|f\|_\infty\le1\}\subset\mathfrak{F}$.
   
    \vspace{2mm}
    
    Since $\mathfrak{F}\supset L_b(\X)$, for any $\varepsilon>0$, $\exists\, f\in\mathfrak{F}$ such that $\chi_{\{B(z,\varepsilon/2)\}}\leq f\leq \chi_{\{B(z,\varepsilon)\}}.$ 
    According to Proposition \ref{Prop:Exist*}, there exists an ergodic invariant probability measure $\mu$ and $\{Q_t(z,\cdot)\}_{t\geq 0}$ weakly converges to $\mu$ as $t\rightarrow\infty$, hence
		\begin{equation}\label{eq 3.9}
		\liminf\limits_{t\rightarrow\infty}Q_t(z,B(z,\varepsilon/2))\geq \mu(B(z,\varepsilon/2)) =:\alpha>0.
		\end{equation}
            Since $\{\Phi_t\}_{t\geq 0}$ is Ces\`aro eventually continuous w.r.t $\mathfrak{F}$ at $z,$ there exists $\delta>0$ such that $\forall\,z'\in B(z,\delta),$
		\begin{equation}\label{eq 3.10}
		\limsup\limits_{t\rightarrow\infty}|Q_tf(z')-Q_tf(z)|<\alpha/4.
		\end{equation}
        
		Combine (\ref{eq 3.9}) and (\ref{eq 3.10}), $\forall\,z'\in B(z,\delta)$,
		\begin{equation}\label{eq 3.11}
		\begin{aligned}
		\liminf\limits_{t\rightarrow\infty}Q_t(z',B(z,\varepsilon))\geq\frac{1}{2} \liminf\limits_{t\rightarrow\infty}Q_t(z,B(z,\varepsilon/2))-\limsup\limits_{t\rightarrow\infty}|Q_tf(z')-Q_tf(z)|\geq \alpha/4.
		\end{aligned}
		\end{equation}
		Let
        \begin{equation*}
		\beta := \inf\limits_{x\in \X}\limsup\limits_{t\rightarrow\infty}Q_t(x,B(z,r))>0.
		\end{equation*}
		For all $x\in \X$, there exists $T$ such that $P_T(x,B(z,\delta))\geq\beta/2$. 
        Define
		\begin{equation*}
		\nu(\cdot):=\P^x(\Phi_T\in\cdot| \Phi_T\in B(z,r)).
		\end{equation*}
		then $\nu(B(z,r))=1$ and 
		Fatou's lemma implies that
		\begin{equation*}
		\liminf\limits_{t\rightarrow\infty}Q_t\nu(B(z,\varepsilon))\geq\int_{X}	\liminf\limits_{t\rightarrow\infty}Q_t(y,B(x,\varepsilon))\nu(\d y)\geq\alpha/4.
		\end{equation*}\par 
		Therefore, we have
		\begin{equation*}
		\liminf\limits_{t\rightarrow\infty}Q_t(x,B(x,\varepsilon))=\liminf\limits_{t\rightarrow\infty}Q_tP_T(x,B(x,r))\geq\alpha\beta/8>0\quad\forall\,x\in \X.
		\end{equation*}
        \noindent 	$\mathbf{Step\;2}.$ The second step is similar to the proof of Theorem \ref{thm:f_mixing}, where we shall apply Lemma \ref{lemma:uniformly_loss_of_memory} below. In view of  Hypothesis $(\mathbf{H_1})$ for the family $\mathfrak{F}\subset B(\X)$, for any $f\in\mathfrak{F}$, let us consider the subfamily $\tilde{\mathfrak{F}}=\{f\}\subset\mathfrak{F}$. Then Hypothesis $(\mathbf{H_2})$ is satisfied for such $\tilde{\mathfrak{F}}$. Moreover, $\Phi$ is uniformly Ces\`aro eventually continuous w.r.t. $\tilde{\mathfrak{F}}=\{f\}$ at $z$. By Lemma \ref{lemma:uniformly_cesaro_loss_of_memory} we have, $\forall\, x \in \X$,
        \begin{equation*}
			\limsup\limits_{t \to \infty} |Q_tf(x) - \langle f,\mu\rangle | = 0,
 	\end{equation*}
    which completes the proof.

\end{proof}

\subsection{Proof of Theorem 
\ref{thm:f_uniform_mixing}}\label{Subsection:proof_of_f_uniform_mixing}

Before turning to the proof of Theorem 
\ref{thm:f_uniform_mixing}, we first present the following technical lemma, which constitutes the key component of the argument. Its proof relies essentially on coupling techniques.

\begin{lemma}\label{lemma:uniformly_loss_of_memory}
	Let $\mathfrak{F}\subset B(\X)$ satisfy Hypothesis $(\mathbf{H_2})$ and $\mu$ be the unique invariant measure of $\Phi$.
    If there exists $z \in \X$ such that 
    \begin{itemize}
        \item [(\runum{1})] $\Phi$ is uniformly eventually continuous w.r.t $\mathfrak{F}$ at $z$;
        \item [(\runum{2})] $\Phi$ satisfies the lower bound condition \eqref{eq:LBC_P} at $z$.
    \end{itemize}
	Then for all $x \in \X$,
    \begin{equation*}
		\limsup\limits_{t \to \infty} \sup_{f\in\mathfrak{F}}|P_tf(x) - \langle f,\mu\rangle | = 0.
	\end{equation*}
\end{lemma}

    \begin{proof}
	By $\pi$-$\lambda$ theorem, a transition probability kernel $R_t$ on the Polish state space $(\X \times \X, \mathcal{B}(\X \times \X))$ is uniquely defined by 
	\begin{equation*}
		R_t((x,y),A \times B) = P_t(x,A) \cdot P_t(y,B), \quad \forall\, A,B \in \mathcal{B}(\X), \quad \forall\, (x,y) \in \X \times \X .
	\end{equation*} 
	Let $\left(\Omega, \mathcal{F},\{\vec{\Phi}_t=(\Phi_t^{(1)},\Phi_t^{(2)})\}_{t\ge 0}, \{\mathcal{F}_t\}_{t\ge 0}, \{\P^{(x,y)}\}_{(x,y)\in {\X\times \X}}\right)$ denote the Markov family with transition kernels $\{R_t\}_{t\ge0}$. Define the probability measure 
	\begin{equation*}
		\P^{\Theta}(\cdot) = \int_{\X\times \X} \P^{(x,y)}(\cdot)\Theta(\d x,\d y)\text{, }\forall\, \Theta \in \mathcal{M}(\X\times \X).
	\end{equation*}
	The corresponding expectation under $\P^\Theta$ is denoted by $\mathbb{E}^{\Theta}$.	
	Let $F(x,y)=f(x)-f(y)$ and $\Theta_0=\delta_x\times\mu$. Since    
    \begin{align*}
		|P_tf(x) - \langle P_tf,\mu\rangle | \
        &= \left| \mathbb{E}^{\Theta_0} F(\vec{\Phi}_t)\right|.
	\end{align*}
    It suffices to show that for all $x \in \X$ and $\varepsilon>0$,
    \begin{equation*}
        \limsup\limits_{t \to \infty} \sup_{f\in\mathfrak{F}} \mathbb{E}^{\Theta_0} |F(\vec{\Phi}_t) | \le 3\varepsilon.
    \end{equation*}
    \noindent $\mathbf{Step\;1. \ Hitting \ Time\ Analysis.}$ 
    
    By uniform eventual continuity w.r.t $\mathfrak{F}$, there exists $r = r(\varepsilon)>0$ such that for all $x,y \in \overline{B(z,r)}$,
	\begin{equation}
    \begin{aligned}
        \limsup\limits_{t \to \infty} \sup_{f\in\mathfrak{F}}| \mathbb{E}^{(x,y)} F(\vec{\Phi}_t) |
        =&\limsup_{t \to \infty} \sup_{f\in\mathfrak{F}}|P_tf(x) - P_tf(y)| \\
        \le&
		\limsup_{t \to \infty} \sup_{f\in\mathfrak{F}}|P_tf(x) - P_tf(z) |+\limsup\limits_{t \to \infty} \sup_{f\in\mathfrak{F}}|P_tf(y) - P_tf(z) |
        < \varepsilon.
    \end{aligned}
		\label{eq:2}
	\end{equation}
	Let $B := \overline{B(z,r)}\times \overline{B(z,r)} \in \mathcal{B}(\X\times \X)$, and $\tau:= \inf\{t\ge 0: \vec \Phi_t\in B\}$. The goal of this step is to show
	\begin{equation}\label{eq:hitting_time_is_as_finite}
		\P^{\Theta_0} (\tau < \infty ) = 1. 
	\end{equation}
    
	By lower bound condition \eqref{eq:LBC_P}, for any initial point $(x,y)\in\X\times\X$, there exists $T^* =T^*(x,y)$ such that for all $t \ge T^*$,  
	\begin{equation*}
		\P^{(x,y)}(\vec{\Phi}_{t}\in B) = P_t(x,B(z,r)) \cdot P_t(y,B(z,r)) \ge \gamma_z(r),
	\end{equation*}
    where
    \begin{equation*}
        \gamma_z(r) = \left(\frac{\inf_{x\in\X} \liminf\limits_{t \to \infty} P_t(x,B(z,r))}2\right)^2>0.
    \end{equation*}
	Therefore,
	\begin{equation*}
		\liminf\limits_{t \to \infty}\P^{(x,y)}(\vec{\Phi}_{t}\in B) \ge \gamma_z(r).
	\end{equation*}
	By Fatou's lemma, for all $\Theta \in \mathcal{P}(\X\times \X)$,
	\begin{equation*}
		\begin{aligned}
			\liminf\limits_{t \to \infty}\P^{\Theta} (\vec{\Phi}_{t}\in B) &= 	\liminf\limits_{t \to \infty} \int_{\X\times\X} \P^{(x,y)}(\vec{\Phi}_t\in B) \Theta(\d x,\d y)\\
			&\ge \int_{\X\times\X} \liminf\limits_{t \to \infty} \P^{(x,y)}(\vec{\Phi}_t\in B) \Theta(\d x,\d y) \ge \gamma_z(r) .
		\end{aligned}
	\end{equation*}
    Hence, $\forall\,\Theta \in \mathcal{P}(\X\times \X)$, there exists $T=T(\Theta)$ such that
		\begin{equation*}
			\P^{\Theta} (\vec{\Phi}_{t}\in B)  \ge\frac{ \gamma_z(r)}{2},
		\end{equation*}
    
    Specifically, for $\Theta_0 = \delta_{x}\times\mu$, take $t_0 = T(\Theta_0)$ then
	\begin{equation*}
		\P^{\Theta_0} (\tau > t_0 )\le\P^{\Theta_0}\left( \vec{\Phi}_{t_0} \notin B\right) \le 1-\frac{\gamma_z(r)}{2}.
	\end{equation*}
	Recursively define
    \begin{equation*}
        \Theta_{i+1}(\cdot) = \P^{\Theta_i}\left(\vec{\Phi}_{t_{i}} \in \cdot | t_i \le \tau\right), \quad t_{i+1} = T(\Theta_i),\quad i=1,2,3\cdots
    \end{equation*}
    Then we have
    \begin{equation*}
		\P^{\Theta_i} (\tau > t_i )  \le  \P^{\Theta_i}\left( \vec{\Phi}_{t_i} \notin B\right) \le 1-\frac{\gamma_z(r)}{2},\quad i\in\Z_+.
	\end{equation*}
	By Markov property,
	\begin{equation*}
        \begin{aligned}
            \P^{\Theta_0} (\tau = \infty) &\le \P^{\Theta_0} (\tau >t_0 + t_1 + \cdots + t_n ) \\
            &\le \P^{\Theta_1} (\tau > t_1 + \cdots + t_n) \cdot \left(1-\frac{\gamma_z(r)}{2}\right)
            \le \dots \le \left(1-\frac{\gamma_z(r)}{2}\right)^n.
        \end{aligned}
	\end{equation*}
	As $ n \to \infty$, we get \eqref{eq:hitting_time_is_as_finite}.

    \vspace{2mm}
    \noindent $\mathbf{Step\;2.\ Decomposition\ and\ Estimatition.}$ 

    We calculate that
    \begin{align*}
    \limsup\limits_{t\to\infty}\sup_{f\in\mathfrak{F}}\left|\mathbb{E}^{\Theta_0}\left[F(\vec{\Phi}_{t})\right]\right|
        \le&\limsup\limits_{t\to\infty}\sup_{f\in\mathfrak{F}}\left|\mathbb{E}^{\Theta_0}\left[F(\vec{\Phi}_{t})\chi_{\{\tau\le t\}}\right]\right|+\limsup\limits_{t\to\infty}\sup_{f\in\mathfrak{F}}\left|\mathbb{E}^{\Theta_0}\left[F(\vec{\Phi}_{t})\chi_{\{\tau>t\}}\right]\right|\\
        := & I +I\!I.
    \end{align*}
    By the strong Markov property,
   \begin{align*}
       I=&\limsup_{t\to\infty} \sup_{f\in\mathfrak{F}}\left|\E^{\Theta_0}\left[F(\vec{\Phi}_{t})\chi_{\{\tau \le t\}}\right]\right|
       \\
       =&\limsup\limits_{t \to \infty} \sup_{f\in\mathfrak{F}}\left|\E^{\Theta_0}\left[\E^{\Theta_0}\left(F(\vec{\Phi}_{t})|\mathcal{F_\tau}\right)\chi_{\{\tau \le t\}}\right]\right|\\
       \le&\limsup\limits_{t \to \infty} \sup_{f\in\mathfrak{F}}\E^{\Theta_0}\left|\E^{\Theta_0}\left[F(\vec{\Phi}_{t})|\mathcal{F_\tau}\right]\chi_{\{\tau \le t\}}\right|\\
       \le&\limsup\limits_{t \to \infty} \E^{\Theta_0}\sup_{f\in\mathfrak{F}}\left|\E^{\Theta_0}\left[F(\vec{\Phi}_{t})|\mathcal{F_\tau}\right]\chi_{\{\tau \le t\}}\right|.
   \end{align*}
   Note that
   \begin{align*}
   \sup_{f\in\mathfrak{F}}\left|\E^{\Theta_0}\left[F(\vec{\Phi}_{t})|\mathcal{F_\tau}\right]\chi_{\{\tau \le t\}}\right|\le&\sup_{f\in\mathfrak{F}}\E^{\Theta_0}\left[\left|F(\vec{\Phi}_{t})\right||\mathcal{F_\tau}\right]\\
   \le& \sup_{f\in\mathfrak{F}}\E^{\Theta_0}\left[\left|f({\Phi}^{(1)}_{t})\right||\mathcal{F_\tau}\right]+\sup_{f\in\mathfrak{F}}\E^{\Theta_0}\left[\left|f({\Phi}^{(2)}_{t})\right||\mathcal{F_\tau}\right]\\
   \le&\E^{\Theta_0}\left[V({\Phi}^{(1)}_{t})|\mathcal{F_\tau}\right]+\E^{\Theta_0}\left[V({\Phi}^{(2)}_{t})|\mathcal{F_\tau}\right]
   \end{align*}
   By Hypothesis $(\mathbf{H_2})$, $\{V({\Phi}^{}_{t})\}_{t\ge0}$ is uniformly integrable under $\P^{x}$ and hence $\left\{V({\Phi}^{(1)}_{t})\right\}_{t\ge0}$ is uniformly integrable under $\P^{\Theta_0}$. Since $\mu$ is invariant measure for $\Phi$, $\{V(\Phi_t)\}_{t\ge0}$ has the same distribution under $\P^\mu$ and $\E^\mu V(\Phi_t)=\langle V,\mu\rangle<\infty$. Therefore, $\left\{V({\Phi}^{(2)}_{t})\right\}_{t\ge0}$ is uniformly integrable under $\P^{\Theta_0}$. By Lemma \ref{lem:UI_mathematic_expactation}, both $\left\{\E^{\Theta_0}\left[V({\Phi}^{(1)}_{t})|\mathcal{F_\tau}\right]\right\}_{t\ge0}$ and $\left\{\E^{\Theta_0}\left[V({\Phi}^{(2)}_{t})|\mathcal{F_\tau}\right]\right\}_{t\ge0}$ are uniformly integrable.
   Applying Lemma \ref{lem:UI_exchange_limit_and_integral} to $I$,
   \begin{align*}
       I\le &\E^{\Theta_0}\limsup\limits_{t \to \infty} \sup_{f\in\mathfrak{F}}\left|\E^{\Theta_0}\left[F(\vec{\Phi}_{t})|\mathcal{F_\tau}\right]\chi_{\{\tau \le t\}}\right|\\
       =&\E^{\Theta_0}\limsup\limits_{t \to \infty} \sup_{f\in\mathfrak{F}}\left|\E^{\vec{\Phi}_{\tau}}\left[F(\vec{\Phi}_{t-\tau})\right]\chi_{\{\tau \le t\}}\right|.
   \end{align*}
    Since  $\mathbb{E}^{\Theta_0}(\vec \Phi_\tau \in {B}) =1$, \eqref{eq:2} and \eqref{eq:hitting_time_is_as_finite},
    \begin{equation*}
        \limsup\limits_{t \to \infty} \sup_{f\in\mathfrak{F}}\left| \mathbb{E}^{\vec{\Phi}_{\tau(\omega)}(\omega)}F(\vec{\Phi}_{t-\tau(\omega)}(\omega)) \chi_{\{\tau(\omega)\le t\}} \right|\le \varepsilon \quad \P^{\Theta}\text{-a.s. } \omega.
    \end{equation*}
    This implies that
    \begin{equation*}
        I\le\varepsilon.
    \end{equation*}
    For the second term, 
    \begin{align*}
       \limsup\limits_{t \to \infty} \sup_{f\in\mathfrak{F}}\left|\E^{\Theta_0}\left[F(\vec{\Phi}_{t})\chi_{\{\tau > t\}}\right]\right|
       \le&\limsup\limits_{t \to \infty} \E^{\Theta_0}\sup_{f\in\mathfrak{F}}\left|F(\vec{\Phi}_{t})\chi_{\{\tau > t\}}\right|\\
       \le&\limsup\limits_{t \to \infty} \E^{\Theta_0}\sup_{f\in\mathfrak{F}}\left|f(\Phi^{(1)}_{t})\chi_{\{\tau > t\}}\right|+\limsup\limits_{t \to \infty} \E^{\Theta_0}\sup_{f\in\mathfrak{F}}\left|f(\Phi^{(2)}_{t})\chi_{\{\tau > t\}}\right|\\
       \le&\limsup\limits_{t \to \infty} \E^{\Theta_0}\left[V(\Phi^{(1)}_{t})\chi_{\{\tau > t\}}\right]+\limsup\limits_{t \to \infty} \E^{\Theta_0}\left[V(\Phi^{(2)}_{t})\chi_{\{\tau > t\}}\right]
   \end{align*}
    For given $\varepsilon>0$, since $\left\{V(\Phi_t^{(1)})\right\}_{t\ge0}$ is uniformly integrable, there exist $\Delta_0\ge0$ such that for any $A\in\mathcal{F}$ satisfying $\P^{\Theta_0}(A)\le\Delta_0$, \begin{equation*}
\E^{\Theta_0}\left[V(\Phi_t^{(1)})\chi_{A}\right]\le\varepsilon.
    \end{equation*}
    Consequently, applying  \eqref{eq:hitting_time_is_as_finite}, there exists $T_0$ such that for $t\ge T_0$,\begin{equation*}
        \P^{\Theta_0}(\tau>t)\le \Delta_0
    \end{equation*} 
    To conclude,
    \begin{equation*}
    \limsup\limits_{t\to\infty}\E^{\Theta_0}\left[V(\Phi_t^{(1)})\chi_{\{\tau>t\}}\right]\le\varepsilon.
    \end{equation*}
    Similarly, since $\left\{V(\Phi_t^{(2)})\right\}_{t\ge0}$ is uniformly integrable under $\P^{\Theta_0}$,
    \begin{equation*}
    \limsup\limits_{t\to\infty}\E^{\Theta_0}\left[V(\Phi_t^{(2)})\chi_{\{\tau>t\}}\right]\le\varepsilon.
    \end{equation*}
    Combining both terms we get $I\!I\le2\varepsilon$.
    As $\varepsilon$ is arbitrary, this completes the proof of Lemma \ref{lemma:uniformly_loss_of_memory}.
	\end{proof}
     
    \paragraph{Proof of Theorem \ref{thm:f_uniform_mixing}:}
	\begin{proof}
	\noindent\emph{(\runum{1}) $\implies$ (\runum{2}):}
	Since $\mathbf{\Phi}$ is uniformly asymptotically stable w.r.t $\mathfrak{F}$, there exists the unique invariant probability measure $\mu\in\mathcal{P}(\X)$ with $\mathfrak{F}\subset L^1(\mu)$. Furthermore, for any $x \in \X$,  
    \begin{equation*}
       \lim\limits_{t\to\infty}\sup_{f\in\mathfrak{F}}\vert P_tf(x) - \langle f, \mu \rangle \vert=0.
    \end{equation*}
    Given $x,y\in\X$,
	\begin{equation*}
		\lim\limits_{t\to\infty}\sup_{f\in\mathfrak{F}}\vert P_tf(x) - P_tf(y)\vert 
        \le \lim\limits_{t\to\infty}\sup_{f\in\mathfrak{F}}\vert P_tf(x) - \langle f, \mu \rangle \vert  + \lim\limits_{t\to\infty}\sup_{f\in\mathfrak{F}}\vert P_tf(y) - \langle f, \mu \rangle \vert =0 ,
	\end{equation*}
	which implies Ces\`aro eventually continuity w.r.t $\mathfrak{F}$ on $\X$.
		
	Finally,  for $z \in \supp \mu$, $r>0$, since $\{f\in L_b(\X):\|f\|_{\infty}\le1\}\subset\mathfrak{F}$ , there exists $f \in \mathfrak{F}$ such that $\chi_{B(z,r/2)} \le f \le \chi_{B(z,r)} $. Hence for all $x \in \X$, 
	\begin{equation*}
		\lim\limits_{t\to \infty} P_t (x, B(z,r))  \ge \lim\limits_{t\to \infty} P_tf(x) = \langle f, \mu \rangle \ge \mu(B(z,r/2)) > 0.
	\end{equation*} 
    \vspace{1mm}
    
	\noindent\emph{(\runum{2}) $\implies$ (\runum{3}):}  This implication is straightforward.\vspace{2mm}	
			
	\noindent\emph{(\runum{3}) $\implies$ (\runum{1}):} For the same reason as Theorem \ref{thm:f_mixing} there exists invariant ergodic probability measure $\mu\in \mathcal{P}(\X)$ and $\mathfrak{F}\subset L_1(\X,\mu)$. 

    Now, $\forall\, x \in \X$,
	By lemma \ref{lemma:uniformly_loss_of_memory}, we get the uniformly asymptotic stability w.r.t $\mathfrak{F}$.
	\end{proof}

\subsection{Proof of Theorem \ref{thm:Q_f_uniform_mixing}}\label{Subsection:proof_of_Q_f_uniform_mixing}

\paragraph{Proof of Theorem \ref{thm:Q_f_uniform_mixing}:}
\begin{proof}
	\noindent\emph{(\runum{1})$\implies$(\runum{2}):}
	Since $\mathbf{\Phi}$ is uniformly mean ergodic w.r.t $\mathfrak{F}$, there exists the unique invariant probability measure $\mu\in\mathcal{P}(\X)$ with $\mathfrak{F}\subset L^1(\mu)$. Furthermore, for any $x \in \X$,  
    \begin{equation*}
   \lim\limits_{t\to\infty}\sup_{f\in\mathfrak{F}}\vert Q_tf(x) - \langle f, \mu \rangle \vert=0.
    \end{equation*}
    
    Given $x,y\in\X$,
	\begin{equation*}
		\lim\limits_{t\to\infty}\sup_{f\in\mathfrak{F}}\vert Q_tf(x) - Q_tf(y)\vert 
        \le \lim\limits_{t\to\infty}\sup_{f\in\mathfrak{F}}\vert Q_tf(x) - \langle f, \mu \rangle \vert  + \lim\limits_{t\to\infty}\sup_{f\in\mathfrak{F}}\vert Q_tf(y) - \langle f, \mu \rangle \vert =0 ,
	\end{equation*}
	which implies Ces\`aro eventually continuity w.r.t $\mathfrak{F}$ on $\X$.\vspace{1mm}
		
	Finally, let $z \in \supp \mu$, $r>0$, since $\{f\in L_b(\X):\|f\|_{\infty}\le1\}\subset\mathfrak{F}$, there exists $f \in \mathfrak{F}$ such that $\chi_{B(z,r/2)} \le f \le \chi_{B(z,r)} $. Hence for all $x \in \X$, 
	\begin{equation*}
		\lim\limits_{t\to \infty} Q_t (x, B(z,r)) \ge \lim\limits_{t\to \infty} Q_tf(x) = \langle f, \mu \rangle \ge \mu(B(z,r/2)) > 0.
	\end{equation*} 
    This verifies the lower bound condition \eqref{eq:LBC_Q} at $z$.
    \vspace{1mm}
    
	\noindent\emph{(\runum{2})$\implies$(\runum{3}):}  This implication is straightforward.\vspace{3mm}	

	\noindent\emph{(\runum{3})$\implies$(\runum{1}):}
        With the same reasoning as in the proof of Theorem~\ref{thm:Q_f_mixing}, the lower bound condition~\eqref{eq:LBC_Q} is equivalent to~\eqref{eq:LBC_3} under the assumption of Cesàro eventual continuity with respect to $\mathfrak{F}$. Moreover, there exists a unique ergodic invariant probability measure $\mu \in \mathcal{P}(\X)$. To complete the proof of Theorem \ref{thm:f_uniform_mixing}, we invoke the following technical lemma.
    
    \begin{lemma}
    \label{lemma:uniformly_cesaro_loss_of_memory}
    Let $\mathfrak{F}\subset B(\X)$ satisfy Hypothesis $(\mathbf{H_2})$ and $\mu$ be the unique invariant measure of $\Phi$.  
    If there exists $z \in \X$ such that 
    \begin{itemize}
        \item [(\runum{1})] $\Phi$ is uniformly Ces\`aro eventually continuous w.r.t $\mathfrak{F}$ at $z$;
        \item [(\runum{2})] $\Phi$ satisfies the lower bound condition $(\mathcal{C}_3)$ at $z$.
    \end{itemize}
		Then for all $x\in \X$,
        \begin{equation*}
			\limsup\limits_{t \to \infty} \sup_{f\in\mathfrak{F}}|Q_tf(x) - Q_tf(z) | = 0.
		\end{equation*}
        and
        \begin{equation*}
			\limsup\limits_{t \to \infty} \sup_{f\in\mathfrak{F}}|\langle f,\mu\rangle - Q_tf(z) | = 0.
		\end{equation*}
    \end{lemma}     
    Invoking Lemma \ref{lemma:uniformly_cesaro_loss_of_memory},  it directly follows that
    \begin{equation*}
			\limsup\limits_{t \to \infty} \sup_{f\in\mathfrak{F}}|Q_tf(x)-\langle f,\mu\rangle | = 0,\quad 
            \forall\, x\in\X,
		\end{equation*}
    which implies the desired result.
\end{proof}

It remains to establish Lemma \ref{lemma:uniformly_cesaro_loss_of_memory}, which is based on a coupling construction in the following.
\paragraph{Proof of Lemma \ref{lemma:uniformly_cesaro_loss_of_memory}}
\begin{proof}
    Let $(\Omega,\mathcal{F},\{\Phi_t\}_{t\ge 0}, \{\mathcal{F}_t\}_{t\ge 0}, \{\P^x\}_{x\in \X})$ be a Markov family with associated Markov operators $\{P_t\}_{t\ge0}$. Define the probability measure 
	\begin{equation*}
		\P^{\nu}(\cdot) := \int_\X \P^{x}(\cdot)\nu(\d x)\quad\forall\, \nu \in \mathcal{P}(\X).
	\end{equation*}
    The expectation taken under $\P^{\nu}$ is denoted by $\E^{\nu}$.
    It suffices to show for any $\varepsilon >0$,
	\begin{equation}\label{goal1}
        \limsup\limits_{t \to \infty} \sup_{f\in\mathfrak{F}}\left|\frac{1}{t}\int_0^t\mathbb{E}^{\nu}f(\Phi_s)\d s - Q_tf(z) \right|\le 5\varepsilon,
	\end{equation}
    where $\nu=\delta_x$ for any $x\in\X$ or $\nu=\mu$.
    The proof is divided into two steps.

    \vspace{2mm}
    
    \noindent $\mathbf{Step\;1. \ Hitting \ Time\ Analysis.}$
    
    Given $\varepsilon >0$, by uniformly Ces\`aro eventual continuity w.r.t $\mathfrak{F}$, there exists $r > 0$ such that for all $ z' \in B(z,r)$, 
    \begin{equation}\label{eq:cesaro_local_loss_of_memory}
		\limsup\limits_{t \to \infty} \sup_{f\in\mathfrak{F}}|Q_tf(z') - Q_tf(z) | \le \varepsilon.	
	\end{equation}
    Define the stopping time $\tau := \inf\{ t \ge 0 : \Phi_t \in \overline{B(z, r)} \}$. The main goal of this step is to show: 
    \begin{equation}
        \P^{\nu} (\tau < \infty) =1,\quad\forall\,\nu \in \mathcal{P}(\X).
        \label{eq:stoping_time_is_finte}
    \end{equation}
    
    In view of lower bound condition \eqref{eq:LBC_3}, define
    \begin{equation*}
        \beta_z(r):= \inf_{x\in\X} \liminf\limits_{t \to \infty} Q_t(x,B(z,r)) >0.
    \end{equation*}
    By Fatou's lemma, for all $\nu \in \mathcal{P}(\X)$,
	\begin{equation*}
		\liminf\limits_{t\to\infty}\int_\X Q_t(x,B(z,r))\nu(\d x)
		\ge\int_\X\liminf\limits_{t \to\infty} Q_t(x,B(z,r))\nu(\d x)\ge\beta_z(r),
	\end{equation*}
	Hence, $\forall\,\nu\in\mathcal{P}(\X)$, there exists $T=T(\nu)$ such that
	\begin{equation*}
		\frac{1}{T}\int_0^T \P^{\nu} (X_{t}\in B(z,r)) \d t \ge\frac12 \beta_z(r),
	\end{equation*}
	which implies 
    \begin{equation*}
        \P^{\nu}(\tau\le T)\ge\sup_{t\in[0,T]} \P^{\nu}(X_{t}\in B(z,r))\ge\frac{1}{T}\int_0^T \P^{\nu}(X_{t}\in B(z,r))\d t\ge\beta_z(r)/2.
    \end{equation*}
    Set $\nu_0 =\nu$, and for $i\in\N$
    recursively define
    \begin{equation*}
         t_{i} :=T(\nu_i);\quad \nu_{i+1}(\cdot) := \P^{\nu_i}(\Phi_{t_{i+1}} \in \cdot | t_i \le \tau).
    \end{equation*}
    By Markov property,
	\begin{equation*}
		\begin{aligned}
		    \P^{\nu} (\tau = \infty) \le& \P^{\nu_0} (\tau >t_0 + t_1 + \cdots + t_n )\\
            \le& \P^{\nu_1} (\tau > t_1 + \cdots + t_n) \cdot \left(1-\frac{\beta_z(r)}{2}\right)\le \dots \le\left( 1-\frac{\beta_z(r)}{2}\right)^n.
		\end{aligned}
	\end{equation*}
    Then relation \eqref{eq:stoping_time_is_finte} is ensured by taking $ n \to \infty$.

\vspace{2mm}

\noindent $\mathbf{Step\;2. \ Decomposition\ and\ Estimation.}$ 

We compute that
    \begin{align*}
    \limsup_{t \to \infty} \sup_{f \in \mathfrak{F}} \left| \mathbb{E}^{\nu} \left[ \frac{1}{t} \int_0^t f(\Phi_s) \, \d s \right] - Q_t f(z) \right|\le&\limsup_{t \to \infty} \sup_{f \in \mathfrak{F}} \left| \mathbb{E}^{\nu} \left[ \frac{1}{t} \int_0^t f(\Phi_s) \, \d s \cdot \chi_{\{\tau \le t\}} \right] - Q_t f(z) \right| \\
    &+ \limsup_{t \to \infty} \sup_{f \in \mathfrak{F}} \left| \mathbb{E}^{\nu} \left[ \frac{1}{t} \int_0^t f(\Phi_s) \, \d s \cdot \chi_{\{\tau > t\}} \right] \right| \\
    := &I + I\!I.
\end{align*}

One useful fact is that the family $\left\{ \frac{1}{t} \int_0^t V(\Phi_s) \, \mathrm{d}s \right\}_{t \ge 0}$ is uniformly integrable under $\mathbb{P}^\nu$ for $\nu=\mu \text{ or }\delta_x$ where $x\in\X$. Actually, by Lemma \ref{lem:UI_Cesaro_av}, it suffice to show the uniform integrability of $\{V(\Phi_t)\}_{t \ge 0}$.

When $\nu = \mu$, the uniform integrability of $\{V(\Phi_t)\}_{t \ge 0}$ is guaranteed by the fact that it is identically distributed under $\mathbb{P}^\mu$, and $\langle V(\Phi_t), \mathbb{P}^\mu \rangle = \langle V, \mu \rangle < \infty$; when $\nu = \delta_x$, the uniformly integrability is guaranteed by Hypothesis $\mathbf{(H_2)}$.

Therefore, for given $\varepsilon>0$, there exists $\Delta>0$ such that for any $A\in\mathcal{F}$ satisfying $\P^{\nu}(A)\le \Delta$, one has
\begin{equation}
    \E^\nu\left\{\frac{1}{t}\int_0^tV(\Phi_s)\d s\cdot \chi_A\right\}\le  \varepsilon.
\end{equation}

Let us fix any $T_0=T_0(\nu)>0$ sufficiently large be such that $\P^\nu(\tau >T_0)<\Delta$. 

For term $I\!I$, 
\begin{equation*}
I\!I \le \limsup_{t \to \infty} \mathbb{E}^\nu \left[ \frac{1}{t} \int_0^t V(\Phi_s) \, \mathrm{d}s \cdot \chi_{\{\tau > T_0\wedge t\}} \right] \le \varepsilon.
\end{equation*}

For term $I$, we decompose it as follows:
\begin{align*}
    I \le{} &\limsup_{t \to \infty} \mathbb{E}^{\nu} \left\{ \sup_{f \in \mathfrak{F}} \left| \mathbb{E}^{\nu} \left[ \frac{1}{t} \int_\tau^{t+\tau} f(\Phi_s) \d s \,\middle|\, \mathcal{F}_\tau \right] \chi_{\{\tau \le t\}} - Q_t f(z) \right| \right\} \\
    &\quad + \limsup_{t \to \infty} \sup_{f \in \mathfrak{F}} \left| \mathbb{E}^{\nu} \left[ \frac{1}{t} \int_0^\tau f(\Phi_s) \d s \cdot \chi_{\{\tau \le t\}} \right] \right| \\
    &\quad + \limsup_{t \to \infty} \sup_{f \in \mathfrak{F}} \left| \mathbb{E}^{\nu} \left[ \frac{1}{t} \int_t^{t+\tau} f(\Phi_s) \d s \cdot \chi_{\{\tau \le t\}} \right] \right| \\
    :=&I_1 + I_2 + I_3.
\end{align*}

Note that
\begin{equation*}
    \sup_{f \in \mathfrak{F}} \left| \mathbb{E}^{\nu} \left[ \frac{1}{t} \int_\tau^{t+\tau} f(\Phi_s) \d s \,\middle|\, \mathcal{F}_\tau \right] \chi_{\{\tau \le t\}} - Q_t f(z) \right| 
    \le 2 \cdot \mathbb{E}^{\nu} \left[ \frac{1}{2t} \int_0^{2t} V(\Phi_s) \d s \,\middle|\, \mathcal{F}_\tau \right] + \sup_{t \ge 0} Q_t V(z).
\end{equation*}
By Lemma \ref{lem:UI_mathematic_expactation}, the family $\left\{ \mathbb{E}^{\nu} \left[ \frac{1}{2t} \int_0^{2t} V(\Phi_s) \d s \,\middle|\, \mathcal{F}_\tau \right] \right\}$ is uniformly integrable, and hence the left-hand side of the above inequality is uniformly integrable.

Applying Lemma \ref{lem:UI_exchange_limit_and_integral} and strong Markov property, we obtain
\begin{align*}
    I_1 &{}\le \mathbb{E}^{\mu} \left\{ \limsup_{t \to \infty} \sup_{f \in \mathfrak{F}} \left| \mathbb{E}^{\mu} \left[ \frac{1}{t} \int_\tau^{t+\tau} f(\Phi_s) \d s \,\middle|\, \mathcal{F}_\tau \right] \chi_{\{\tau \le t\}} - Q_t f(z) \right| \right\} \\
    &={} \mathbb{E}^{\mu} \left\{ \limsup_{t \to \infty} \sup_{f \in \mathfrak{F}} \left| \mathbb{E}^{\Phi_\tau} \left[ \frac{1}{t} \int_0^t f(\Phi_s) \d s \right] \chi_{\{\tau \le t\}} - Q_t f(z) \right| \right\}.
\end{align*}

Thanks to \eqref{eq:cesaro_local_loss_of_memory} and the fact that $\P^\mu(\Phi_\tau \in B(z, r)) = 1$, we have
\begin{align*}
    &\limsup_{t \to \infty} \sup_{f \in \mathfrak{F}} \left| \mathbb{E}^{\Phi_\tau} \left[ \frac{1}{t} \int_0^t f(\Phi_s) \d s \right] \chi_{\{\tau \le t\}} - Q_t f(z) \right| \\
    &= \limsup_{t \to \infty} \sup_{f \in \mathfrak{F}} \left| \mathbb{E}^{\Phi_\tau} \left[ \frac{1}{t} \int_0^t f(\Phi_s) \d s \right] - Q_t f(z) \right| \\
    &= \limsup_{t \to \infty} \sup_{f \in \mathfrak{F}} \left| Q_t f(\Phi_\tau) - Q_t f(z) \right| \le \varepsilon, \quad \P^\mu\text{-a.s.}
\end{align*}
This implies that $I_1 \le \varepsilon$.

 In what follows, let us deal with the remaining terms $I_2$ and $I_3$. Recall that $\left\{\frac{1}{t}\int_0^tV(\Phi_s)\d s\right\}_{t\ge0}$ is uniformly integrable. Thus 
\begin{align*}
    I_2&\le \limsup_{t\ge T_0,t\to\infty}\E^\nu\left[\frac{1}{t}\int_0^{\tau}V(\Phi_s)\chi_{\{\tau\leq t\}}\chi_{\{\tau\leq T_0\}}\d s\right]+\limsup_{t\ge T_0,t\to\infty}\E^\nu\left[\frac{1}{t}\int_0^{\tau}V(\Phi_s)\chi_{\{\tau\leq t\}}\chi_{\{\tau> T_0\}}\d s\right]\\
    &\le \limsup_{t\ge T_0,t\to\infty}\E^\nu\left[\frac{1}{t}\int_0^{T_0}V(\Phi_s)\d s\right]+\limsup_{t\ge T_0,t\to\infty}\E^\nu\left[\frac{1}{t}\int_0^{t}V(\Phi_s)\d s\chi_{\{\tau> T_0\}}\right]\\    
    &=\limsup_{t\to\infty}\left[\frac{1}{t}\int_0^{t}\E^\nu V(\Phi_s)\d s\right]+\varepsilon.
\end{align*}
And similarly, 
\begin{align*}
    I_3&\le \limsup_{t\ge T_0,t\to\infty}\E^\nu\left[\frac{1}{t}\int_t^{t+T_0}V(\Phi_s)\d s\right]+2\cdot\limsup_{t\ge T_0,t\to\infty}\E^\nu\left[\frac{1}{2t}\int_0^{2t}V(\Phi_s)\d s\chi_{\{\tau> T_0\}}\right]\\
    &= \limsup_{t\to\infty}\E^\nu\left[\frac{1}{t}\int_t^{t+T_0}V(\Phi_s)\d s\right]+2\cdot\limsup_{t\to\infty}\E^\nu\left[\frac{1}{2t}\int_0^{2t}V(\Phi_s)\d s\chi_{\{\tau> T_0\}}\right]\\
   & \le \limsup_{t\to\infty}\left[\frac{1}{t}\int_t^{t+T_0}\E^\nu V(\Phi_s)\d s\right]+2\varepsilon.
\end{align*}

When $\nu=\mu$,
\begin{align*}    &\limsup_{t\to\infty}\left[\frac{1}{t}\int_0^{T_0}\E^\nu V(\Phi_s)\d s\right]=\limsup_{t\to\infty}\left[\frac{1}{t}\int_t^{t+T_0}\E^\nu V(\Phi_s)\d s\right]=\limsup_{t\to\infty}\frac{T_0}{t}\cdot\langle V,\mu\rangle=0.
\end{align*}

When $\nu=\delta_x$, \begin{equation*}
    \limsup_{t\to\infty}\left[\frac{1}{T_0}\int_0^{t}\E^\nu V(\Phi_s)\d s\right]=\limsup_{t\to\infty}\left[\frac{1}{t}\int_t^{t+T_0}\E^\nu V(\Phi_s)\d s\right]\le\limsup_{t\to\infty}\frac{T_0}{t}\cdot\sup_{t\ge0}P_tV(x)=0.
\end{equation*}

To conclude, $I\le I_1+I_2+I_3\le4\varepsilon$ and \begin{equation*}
    \limsup_{t \to \infty} \sup_{f \in \mathfrak{F}} \left| \mathbb{E}^{\nu} \left[ \frac{1}{t} \int_0^t f(\Phi_s) \, \d s \right] - Q_t f(z) \right|\le I+I\!I\le5\varepsilon,
\end{equation*}
which completes the proof.

\end{proof} 

\subsection{Proof of Proposition \ref{lem:S_Lypv}}
\begin{proof}
    Since $\mathcal{L}$ is the generator of Feller process $\Phi$ (see \cite[{Definition 3.2.1}]{Kulik2018}) ,
    \begin{equation*}  V(\Phi_t)=V(\Phi_0)+\int_0^t\mathcal{L}V(\Phi_s)\d s + M_t,
    \end{equation*} 
    where $M_t$ is a local martingale. Then there exists a sequence of nondeceasing random variables $\{T_n\}_{n=1}^\infty$ with $\lim\limits_{n\to\infty}T_n=\infty$ a.s. such that $M_{t\wedge T_n}$ is a martingale for all $n\in\N$. Hence
    \begin{equation*}
        \E^x V(\Phi_{t\wedge T_n})\le  V(\Phi_0)-\E^x\int_0^{t\wedge T_n}(\varphi(V(\Phi_s))\d s+C\E^x[t\wedge T_n].
    \end{equation*}
    Since $\lim\limits_{t\to\infty}T_n=\infty$, $\varphi(V(\Phi_s))\ge0$, we have 
    \begin{equation*}
        \E^x V(\Phi_{t})\le\liminf\limits_{n\to\infty}\E^x V(\Phi_{t\wedge T_n})\le\limsup\limits_{n\to\infty}\E^x V(\Phi_{t\wedge T_n})\le  V(\Phi_0)-\E^x\int_0^t(\varphi(V(\Phi_s))\d s+Ct
    \end{equation*}

    Since $\varphi$ is increasing and concave, $\varphi^{-1}$ exists and it is convex. Define $U=\varphi\circ V$, then $V=\varphi^{-1}\circ U$.

    By Jensen's inequality, we have
    \begin{equation*}
        \varphi^{-1}(\E^x U(\Phi_t))\le  E^xV(\Phi_t) \le V(\Phi_0)-\int_0^t(\E^x\varphi(V(\Phi_s))\d s+Ct=  V(\Phi_0)-\int_0^t(\E^xU(\Phi_s))\d s+Ct.
    \end{equation*}
     By comparison principle, since the solution $f(t)\in C^1(\R_+)$ of the following ordinary differential equation:
    \begin{equation*}
    \left\{
        \begin{aligned}
        \frac{df(t)}{dt}&={(-f(t)+C)}{(\varphi'\circ\varphi^{-1}(f(t)))}\\
        f(0)&=U(x)
    \end{aligned}\right.
    \end{equation*}
    converges to some fixed point, i.e. $\lim\limits_{t\to\infty}f(t)<\infty$, hence $\E^x V(\Phi_t)$ is uniformly bounded for all $x\in\X$.

    To be precise, this is because $h(x):={(-x+c)}{(\varphi'\circ\varphi^{-1}(x))}$ has only one fixed point $x=c$ and when $0<x< C$, $h(x)>0$ and when $x> C$, $h(x)<0$. This implies that the fix point is stable.

\end{proof}

    \vspace{5mm}
    \textbf{Acknowledgments.}  \quad  This work is supported by National Natural Science Foundation of China (No.12231002). The authors wish to express our sincere thanks to Professor Yong Liu, our supervisor, for many enlightening discussions and persistent instructions.
 \appendix
 \section{Properties of uniformly integrable random variables}\label{appendix:UI}
 In this appendix, we introduce the properties of uniformly integrable random variables which are useful in the proof of our main theorem.
  Let $\{\xi_t\}_{t\ge0}$ be a family of random variables on the probability space $(\Omega,\mathcal{F},\P)$.
  \begin{definition}
      $\{\xi_t\}_{t\ge0}$ is said to be unifromly integrable, if 
      \begin{equation*}
          \lim_{K\to\infty}\sup_{t\ge0}\E(|\xi_t|\chi_{\{|\xi_t|\ge K\}})=0.
      \end{equation*}
  \end{definition}

  \begin{lemma}\label{lem:boundedness_implies_UI}
    Let $\{\xi_t\}_{t\ge 0}$ be a family of random variables and $g:\R_+\rightarrow\R_+$ satisfying $\lim\limits_{a\to\infty}g(a)/a=\infty$. If 
    \begin{equation}\label{eq:bounded_in_L1}
        \sup_{t\ge0}\E(g(|\xi_t|))<\infty,
    \end{equation}
then $\{\xi_t\}_{t\ge 0}$ is uniformly integrable.
\end{lemma}
See e.g. \cite[Theorem 4.6.2]{Durrett2019} for a proof.\vspace{2mm}

\begin{lemma}
\label{lem:UI_compare}
    Let $\{\eta_t\}_{t\ge0}$ be another family of random variables. Suppose $|\xi_t|\le |\eta_t|$ and $\{\eta_t\}_{t\ge0}$ is uniformly integrable, then $\{\xi_t\}_{t\ge0}$ is uniformly integrable.
\end{lemma}
\begin{proof}
    Since $|\xi_t|\le |\eta_t|$, and $\{|\xi_t|> K\}\subset \{|\eta_t|>K\}$,
    \begin{equation*}
        \E(|\xi_t|\chi_{\{|\xi_t|\ge K\}})\le \E(|\eta_t|\chi_{\{\eta_t|\ge K\}}).
    \end{equation*}
    By definition of uniform integrability, the proof is completed.
\end{proof}
This lemma states that the family of random variables controlled by a family of uniformly integrable random variables is itself uniformly integrable.
\begin{lemma}
    $\{\xi_t\}_{t\ge0}$ is uniformly integrable is equivalent to $\{|\xi_t|\}_{t\ge0}$  is uniformly integrable.
\end{lemma}
\begin{proof}
 By definition,
 \begin{align*}
        &\{\xi_t\}_{t\ge0}\text{ is uniformly integrable }\\\iff&\lim\limits_{K\to\infty} \sup_{t\ge0}\E(|X_t|\chi_{\{|X_t|> K\}})=0\\
        \iff  &\{|\xi_t|\}_{t\ge0}\text{ is uniformly integrable }.
    \end{align*}
\end{proof}
This lemma show that the sign of random variables does not change the uniform integrability.
\begin{lemma}\label{lem:UI_mathematic_expactation}
    Let $\{\xi_t\}_{t\ge0}$ be a family of uniformly integrable random variables and $\mathcal{G}$ be a sub $\sigma$-algebra of $\mathcal{F}$. Then $\{\E[\xi_t|\mathcal{G}]\}_{t\ge 0}$ is uniformly integrable.
\end{lemma}
\begin{proof}
    By Jensen's inequality, 
    \begin{equation*}
        |\E [\xi_t|\mathcal{G}]|\le \E [|\xi_t||\mathcal{G}].
    \end{equation*}
    By Lemma \ref{lem:UI_compare}, without loss of generality we can assume $\xi_t\ge0$ for all $t\ge0$.
    \begin{align*}
        &\E\left\{\E[\xi_t|\mathcal{G}]\chi_{\{\E[\xi_t|\mathcal{G}]> K\}}\right\}\\
        =&\E\left\{\E[\xi_t\chi_{\{\E[\xi_t|\mathcal{G}]>K\}}\right\}|\mathcal{G}]\\
        =&\E[\xi_t\chi_{\{\E[\xi_t|\mathcal{G}]> K\}}].
    \end{align*}
    Since 
    \begin{equation*}
        \P(\E[\xi_t|\mathcal{G}]> K)\le \frac{\sup_{t\ge0}\E[\xi_t]}{K}.
    \end{equation*}
    Because $\{\xi_t\}$ is uniformly integrable and hence uniformly absolutely continuous and uniformly bounded in $L^1$, we have
    \begin{equation*}
        \lim\limits_{K\to\infty}\E\left\{\E[\xi_t|\mathcal{G}]\chi_{\{\E[\xi_t|\mathcal{G}]> K\}}\right\}=0.
    \end{equation*}
\end{proof}
This lemma indicates that uniformly integrability remains after taking mathematic conditional expectation.

\begin{lemma}\label{lem:UI_Cesaro_av}
    Suppose $\{\xi_t\}_{t\ge0}$ is uniformly integrable, then $\{\frac{1}{t}\int_0^t\xi_s\d s\}_{t\ge0}$ is uniformly integrable.
\end{lemma}
\begin{proof}
    Without loss of generality, assume $\xi\ge0$ for all $t\ge0$. 
    
    First, we show it is uniformly bounded in $L^1$:
    \begin{equation*}
        \E\left\{\frac{1}{t}\int_0^t\xi_s\d s\right\}=\frac{1}{t}\int_0^t\E\xi_s\d s\le\sup_{t\ge0}\E\xi_t<\infty.
    \end{equation*}
    The last inequality is due to the uniform integrability of $\{\xi_t\}_{t\ge0}$.
    
    Next, we show it is uniformly absolutely continuous: 
    Given $A\in\mathcal{F}$, by uniformly absolutely continuous of $\{\xi_t\}_{t\ge0}$, given $\varepsilon>0$, there exists $\Delta>0$, when $\P(A)\le\Delta$, $\sup_{t\ge0}\E(\xi\chi_A)\le\varepsilon$. Hence
    \begin{equation*}
        \E\left\{\frac{1}{t}\int_0^t\xi_s\d s\chi_A\right\}=\frac{1}{t}\int_0^t\E(\xi_s\chi_A)\d s\le \varepsilon.
    \end{equation*}
\end{proof}
This lemma shows that the Ces\`aro average of a family of uniformly integrable random variables is uniformly integrable.

Finally, we show the usefulness of uniform integrability in exchanging the integral and limit process.
{\begin{lemma}\label{lem:UI_exchange_limit_and_integral}
    If $\{\xi_t\}_{t\ge0}$ is nonnegative and uniformly integrable, then
    \begin{equation*}
        \limsup_{t\to\infty}\E \xi_t\le \E \limsup_{t\to\infty} \xi_t.
    \end{equation*}
    
\end{lemma}}
\begin{proof}
    For all $\varepsilon>0$, take $K$ such that \begin{equation*}
        \sup_{t\ge0}\E(\xi_t\chi_{\{\xi_t>K\}})<\varepsilon.
    \end{equation*}
    Then \begin{align*}
        \limsup_{t\to\infty}\E \xi_t\le &\limsup_{t\to\infty}\E \xi_t\chi_{\{\xi_t\le K\}}+\varepsilon\\
        \le& \E \limsup_{t\to\infty}\xi_t\chi_{\{\xi_t\le K\}} +\varepsilon\\
        \le&\limsup_{t\to\infty}\xi_t +\varepsilon.
    \end{align*}
    The second inequality is due to the upper bound of $\{\xi_t\chi_{\{\xi_t\le K\}}\}$ and Fatou's lemma. Since $\varepsilon$ is arbitrary, we get the desired results.
\end{proof}
 \section{Generalized Birkhoff's theorem}
 \begin{lemma}[Generalized Birkhoff's ergodic theorem]
\label{lem:generalized_Birk}
    Let $\Phi$ be a Markov process on a probability space $(\Omega,\mathcal{F},\P)$ with unique invariant measure $\mu$. Let $\P^\mu$ be the distribution of $\Phi$ with initial distribution $\mu$. Suppose $f\ge0$ satisfying $\langle f,\mu\rangle=\infty$ then $\P^\mu$-a.s. $\omega\in\Omega$. 
    \begin{equation*}
        \lim\limits_{t\to\infty}\frac{1}{t}\int_0^t f(\Phi_s(\omega))\d s =\infty.
    \end{equation*}
\end{lemma}
\begin{proof}
    Let $f_n(x)=f(x)\chi_{|f(x)|\le n}$. Then $\|f_n\|_{\infty}\le n$. By Birkhoff's ergodic theorem, for any $n\in\N_+$, there exist $\Omega_n\in\Omega$ with $\P^{\mu}(\Omega_n)=1$ such that for all $\omega\in\Omega_n$,
        \begin{equation*}
        \lim\limits_{t\to\infty}\frac{1}{t}\int_0^t f_n(\Phi_s(\omega))\d s =\langle f_n,\mu\rangle.
    \end{equation*}
    Let $\Omega'=\bigcap_{n\in\N_+}\Omega_n$, then $\P^\mu(\Omega')=1$ and for all $\omega\in\Omega'$,
    \begin{equation*}
        \lim\limits_{t\to\infty}\frac{1}{t}\int_0^t f_n(\Phi_s(\omega))\d s =\langle f_n,\mu\rangle,\quad \forall\, n\in\N_+.
    \end{equation*}

    Now, fix $\omega\in\Omega'$. By monotone convergence theorem, $\lim\limits_{n\to\infty}\langle f_n,\mu\rangle=\langle f,\mu\rangle=\infty $. Hence for any $M>0$, there exists $N$ such that for all $n\ge N$, $\langle f_n,\mu\rangle\ge2M$.

    Then for $n=N$, there exists $T(\omega,N)$ such that for all $t\ge T(\omega, N)$, \begin{equation*}
        \frac{1}{t}\int_0^t f_{N}(\Phi_s(\omega))\d s \ge\frac{1}{2}\langle f_{N},\mu\rangle\ge M.
    \end{equation*}
    Since $0\le f_{N}\le f$, 
    \begin{equation*}
        \frac{1}{t}\int_0^t f(\Phi_s(\omega))\d s\ge \frac{1}{t}\int_0^t f_{N}(\Phi_s(\omega))\d s \ge M,\quad \forall\, t\ge T(\omega,N).
    \end{equation*}
    Therefore, 
    \begin{equation*}
         \liminf\limits_{t\to\infty}\frac{1}{t}\int_0^t f(\Phi_s(\omega))\d s \ge M.
    \end{equation*}
    Since $M$ is arbitrary, we have
    \begin{equation*}
         \lim\limits_{t\to\infty}\frac{1}{t}\int_0^t f(\Phi_s(\omega))\d s =\infty, \quad \omega\in\Omega'.
    \end{equation*}
    This completes the proof.
\end{proof}

\small    
\bibliographystyle{abbrv}
\bibliography{References}
\end{document}